\documentclass{amsart}
\usepackage{amsrefs}
\usepackage{mathrsfs}
\usepackage{amsfonts}
\usepackage[centertags]{amsmath}
\usepackage{amssymb}
\usepackage{amsthm}
\usepackage{graphicx}
\usepackage{caption}
\usepackage{hyperref}
\usepackage{enumerate}
\usepackage{geometry}
\usepackage{appendix}
\usepackage{color}
\usepackage[all]{xy}
\usepackage{float}
\usepackage{amsmath,amsfonts,amssymb,graphicx}
\usepackage{longtable}
\usepackage{resizegather}
\usepackage{framed}
\usepackage{tikz}
\usepackage{multirow}
\usepackage{mathdots}

\newtheorem{theorem}{Theorem}[section]
\newtheorem{corollary}[theorem]{Corollary}

\newtheorem{proposition}[theorem]{Proposition}
\newtheorem{definition}[theorem]{Definition}
\newtheorem{remark}[theorem]{Remark}
\newtheorem{example}[theorem]{Example}

\newtheorem{conjecture}[theorem]{Conjecture}

\numberwithin{equation}{section}

\DeclareMathOperator*{\Hom}{\text{Hom}}

\usetikzlibrary{arrows,automata,positioning}
\newcommand{\midarrow}{\tikz \draw[thin,-angle 45] (0,0) -- +(.25,0);}

\title{A geometric $q$-character formula for snake modules}
\author{Bing Duan} 
\address{School of Mathematics and Statistics, Lanzhou University, Lanzhou 730000, P. R. China.}
\email{duan890818@gmail.com}
\thanks{The first author was supported by China Scholarship Council as a Joint PHD student to visit Department of Mathematics at the University of Connecticut and he would like to thank Department of Mathematics for hospitality during his visit. He was also partially supported by the National Natural Science Foundation of China (Grant No. 11771191). }
\author{Ralf Schiffler}
\address{Department of Mathematics, University of Connecticut, Storrs, CT 06269-3009, USA}
\email{schiffler@math.uconn.edu}
\thanks{The second author was supported by NSF CAREER Grant DMS-1254567, NSF Grant DMS-1800860 and by the University of Connecticut.}
\begin{document}

\maketitle

\begin{abstract}  Let $\mathscr{C}$ be the category of finite dimensional modules over the quantum affine algebra $U_q(\widehat{\mathfrak{g}})$ of a simple complex Lie algebra ${\mathfrak{g}}$. Let $\mathscr{C}^-$ be the subcategory introduced by Hernandez and Leclerc. We prove the geometric $q$-character formula conjectured by Hernandez and Leclerc in types $\mathbb{A}$ and $\mathbb{B}$ for a class of simple modules called snake modules introduced by Mukhin and Young. Moreover, we give a combinatorial formula for the $F$-polynomial of the generic kernel associated to the snake module. As an application, we show that snake modules correspond to cluster monomials with square free denominators and we show that snake modules are real modules. We also show that the cluster algebras of the category $\mathscr{C}_1$ are factorial for Dynkin types $\mathbb{A,D,E}$.
%
%
\end{abstract}

\section{Introduction}
Let $\mathfrak{g}$ be a simple complex Lie algebra and let $U_q(\widehat{\mathfrak{g}})$ be the corresponding quantum affine algebra with quantum parameter $q\in \mathbb{C}^\times$ not a  root of unity. 
Denote by $\mathscr{C}$ the category of finite dimensional $U_q(\widehat{\mathfrak{g}})$-modules. The simple modules in $\mathscr{C}$ have been classified in \cite{CP91,CP94} by Chari and Pressley in terms of Drinfeld polynomials. In \cite{FR98}, Frenkel and Reshetikhin attached a $q$-character to every module in $\mathscr{C}$ and showed that the simple modules are determined up to isomorphism by their $q$-characters. 
Moreover, the  simple modules are parametrized by the highest dominant monomials in their $q$-characters. 

Cluster algebras were introduced in \cite{FZ02} by Fomin and Zelevinsky as a tool for studying canonical bases in Lie theory. A cluster algebra is a commutative algebra with a distinguished set of generators, the \emph{cluster variables}. These cluster variables are constructed by a recursive method called mutation, which is determined by the choice of a quiver $Q$ without loops and 2-cycles. Given a cluster algebra $\mathscr{A}(Q)$, every cluster variable can be expressed as a Laurent polynomial with integer coefficients with respect to any given cluster \cite{FZ02} and this Laurent polynomial has positive coefficients \cite{LS15}. A \emph{cluster monomial} is a product of cluster variables from the same cluster. It was proved in \cite{CKLP13} that the set of all cluster monomials is linearly independent.

\subsection{Category $\mathscr{C}_1$} A connection between representations of quantum affine algebras and cluster algebras was discovered by Hernandez and Leclerc in \cite{HL10}, where a monoidal categorification of certain cluster algebras was given. One significant aspect of a monoidal categorification is that, if it exists, it implies the positivity of the cluster variables and the linear independence of the cluster monomials, see \cite[Proposition 2.2]{HL10}.

The monoidal categorification is realized as a subcategory of the category $\mathscr{C}$ as follows. Let $I$ be the vertex set of the Dynkin diagram of $\mathfrak{g}$ 
for Dynkin types $\mathbb{A,D,E}$ and $I=I_0\cup I_1$ be a partition of $I$ such that every edge connects a vertex of $I_0$ with a vertex of $I_1$. Let $\mathscr{C}_\ell$, $\ell\geq 0$, be the full subcategory of $\mathscr{C}$ whose objects $V$ satisfy the following property. For any composition factor $S$ of $V$ and every $i\in I$, the roots of the $i$-th Drinfeld polynomial of $S$ belong to $\{q^{-2k-\xi_i} \mid 0\leq k\leq \ell\}$, where $\xi_i=1$ if $i\in I_1$ and $\xi_i=0$ if $i\in I_0$.

For Dynkin types $\mathbb{A}$ and $\mathbb{D}_4$, it has been shown in \cite{HL10} that the category $\mathscr{C}_1$ is a monoidal categorification of  a cluster algebra of the same Dynkin type. This result was extended to Dynkin types $\mathbb{A,D,E}$ by Nakajima in \cite{Nak11}, see also \cite{HL13}. In \cite{Q17}, Qin proved that every cluster monomial corresponds to a simple module in $\mathscr{C}$ for Dynkin types $\mathbb{A,D,E}$.

A simple module $M$ in $\mathscr{C}$ is said to be \emph{real} if $M\otimes M$ is simple \cite{Le03}, and $M$ is said to be \emph{prime} if it cannot be written as a non-trivial tensor product of modules \cite{CP97}. 

\subsection{Category $\mathscr{C}^-$}In \cite{HL16}, Hernandez and Leclerc considered a much larger subcategory $\mathscr{C}^-$ of $\mathscr{C}$ which contains, up to spectral shifts, all the simple finite-dimensional $U_q(\widehat{\mathfrak{g}})$-modules. They showed that the Grothendieck ring of $\mathscr{C}^-$ has a cluster algebra structure \cite[Theorem 5.1]{HL16}, and they proposed two conjectures. 

\begin{conjecture}
\cite[Conjecture 5.2]{HL16} \label{conj 5.2}
 The cluster monomials of the cluster algebra are in bijection with the isomorphism classes of real simple objects in $\mathscr{C}^-$.
\end{conjecture}

The second conjecture uses the theory of quivers with potentials developed in \cite{DWZ08,DWZ10}. In \cite[Section 5.2.2]{HL16}, Hernandez and Leclerc associated to every simple $U_q(\widehat{\mathfrak{g}})$-module $M$ a so-called \emph{generic kernel} $K(M)$, which is a module over the Jacobian algebra of the quiver with potential. They showed that, up to normalization, the truncated $q$-character of a Kirillov-Reshetikhin module is equal to the $F$-polynomial of the associated generic kernel, and they conjectured the following generalization. 

\begin{conjecture}
\cite[Conjecture 5.3]{HL16} \label{conj 5.3}
Up to normalization, the truncated $q$-character of a real simple module in $\mathscr{C}^-$ is equal to the $F$-polynomial of the associated generic kernel.
\end{conjecture}

\subsection{Snake modules}
In this paper, we prove both conjectures in Dynkin types $\mathbb{A}$ and $\mathbb{B}$ for \emph{snake modules}, a class of simple $U_q(\widehat{\mathfrak{g}})$-modules introduced by Mukhin and Young in \cite{MY12a,MY12b}. In \cite{MY12a}, they introduced  a purely combinatorial method to compute $q$-characters for snake modules of types $\mathbb{A}$ and $\mathbb{B}$, and in \cite{MY12b}, they used snake modules to construct extended $T$-systems for types $\mathbb{A}$ and $\mathbb{B}$. In~\cite{DLL19}, it was shown that all prime snake modules are real and that they correspond to some cluster variables in the cluster algebra constructed by Hernandez and Leclerc.

Our first main theorem is the following.
\begin{theorem} (Theorem \ref{theorem of geometric formula for sm} and Remark \ref{snake modules geometric character formula remark})
Let $L(m)$ be a prime snake module in $\mathscr{C}^-$. Then up to normalization,  the truncated $q$-character of $L(m)$ is equal to the $F$-polynomial of the associated generic kernel $K(m)$. More precisely, 
\begin{align*}
\chi^-_q(L(m)) = m F_{K(m)}.
\end{align*}
 Replacing the module $K(m)$ by a direct sum, we obtain a similar geometric character formula for arbitrary snake module of types $\mathbb{A}$ and $\mathbb{B}$.
\end{theorem}
This proves Conjecture \ref{conj 5.3} for snake modules and gives a geometric algorithm for the truncated $q$-characters. As a slight generalization of Theorem 3.4 of \cite{DLL19}, we show in Theorem \ref{SM-real modules} that snake modules are real modules. 

Then in Theorem~\ref{F-polynomial formula} and Remark \ref{compute dimensional vector}~(3), we give a combinatorial formula for the $F$-polynomial of the generic kernel $K(m)$ associated to a snake module $L(m)$ as a sum over certain non-overlapping paths in a subset of  the $\mathbb{Z}\times\mathbb{Z}$-grid determined by $m$. This result uses the model of Mukhin and Young \cite{MY12a,MY12b}. As a consequence, we also obtain a combinatorial method to find the dimension vector of $K(m)$ as well as all its submodules. Furthermore we show that $K(m)$ is always rigid and it is indecomposable if the snake module $L(m)$ is prime.

As an application, we prove Conjecture \ref{conj 5.2} for snake modules in the first part of the following theorem.

\begin{theorem}(Theorems \ref{cluster monomials with square free denominator} and \ref{denominator formula})
The truncated $q$-character of a snake module $L(m)$ is a cluster monomial. Moreover the denominator of this cluster monomial is square free and is parametrized by the support of $K(m)$ as a representation of the quiver with potential.
\end{theorem}

It is natural to ask whether all cluster variables with square free denominators correspond to snake modules. This is not the case. However, the only counter examples we found are modules whose truncated $q$-characters are not equal to their ordinary $q$-characters. 

Lastly, the study of square free denominators led us to questions of factorization in the cluster algebra. Applying the results of \cite{ELS18}, we include a proof that the $\mathscr{C}_1$ cluster algebras of Hernandez-Leclerc are factorial for Dynkin types $\mathbb{A,D,E}$. 

This paper is organized as follows. In Section \ref{preliminaries}, we briefly review basic materials on cluster algebras, quantum affine algebras, snake modules, and Hernandez and Leclerc's results. Section~\ref{GFSM} is on our geometric character formula for snake modules and Section \ref{denominator vector} is devoted to the study of denominator vectors of cluster monomials corresponding to snake modules. In the last section, we prove that the $\mathscr{C}_1$ cluster algebras are factorial for Dynkin types $\mathbb{A,D,E}$.

\section{Preliminaries}\label{preliminaries}

\subsection{Quivers and cluster algebras}\label{section2.1}
We recall the definition of the cluster algebras introduced in \cite{HL16}. Let $C=(c_{ij})_{i,j\in I}$ be an indecomposable $n\times n$ Cartan matrix of finte type. Then there exists a diagonal matrix $D=\text{diag}(d_i\mid i\in I)$ with positive entries such that $B=DC=(b_{ij})_{i,j\in I}$ is symmetric. Let $t=\max\{d_i\mid i\in I\}$. Thus 
\begin{align*}
t=\begin{cases}
1 & \text{if $C$ is of type $\mathbb{A}_n$, $\mathbb{D}_n$, $\mathbb{E}_6$, $\mathbb{E}_7$, $\mathbb{E}_8$},\\
2 & \text{if $C$ is of type $\mathbb{B}_n$, $\mathbb{C}_n$ or $\mathbb{F}_4$}, \\
3 & \text{if $C$ is of type $\mathbb{G}_2$}.
\end{cases}
\end{align*}

Let $\widetilde{G}$ be the infinite quiver with vertex set $\widetilde{G}_0=I\times \mathbb{Z}$ and arrows $(i,r) \to (j,s)$ if $b_{ij}\neq 0$ and $s=r+b_{ij}$. It needs to be pointed out that $\widetilde{G}$ has two isomorphic connected components, see Lemma 2.2 of \cite{HL16}. We pick one of the two components and denote it by $G$ with vertex set $G_0$. We consider the full subquiver $G^-$ of $G$ with vertex set $G^-_0=G_0\cap (I\times \mathbb{Z}_{\leq 0})$, see Figure \ref{initial quivers A3 and B2}. 

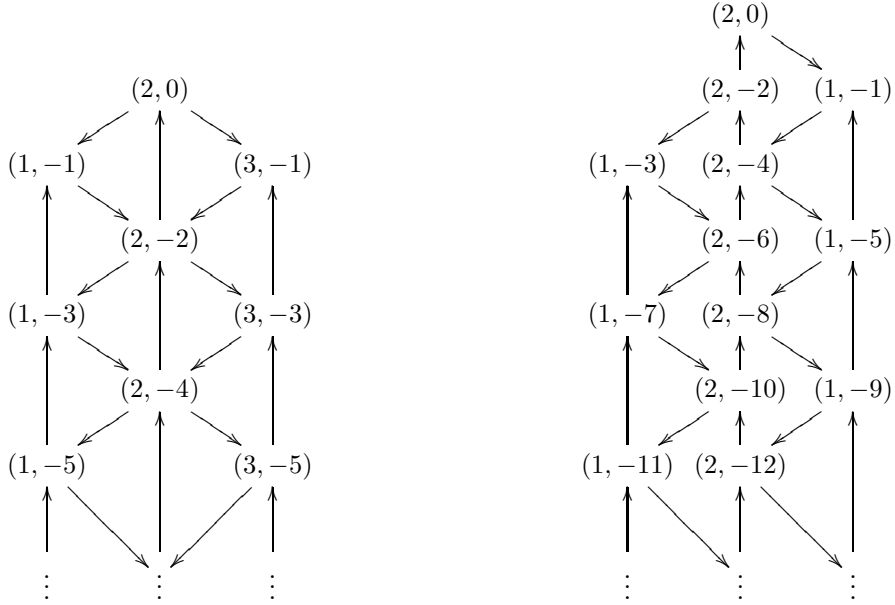
\begin{figure}
\resizebox{1.0\width}{1.0\height}{
\begin{minipage}[t]{.5\linewidth}
\centerline{
\begin{xy}
(25,50)*+{(2,0)}="a";%
(10,40)*+{(1,-1)}="b"; (40,40)*+{(3,-1)}="c";%
(25,30)*+{(2,-2)}="d";%
(10,20)*+{(1,-3)}="e"; (40,20)*+{(3,-3)}="f";%
(25,10)*+{(2,-4)}="g";%
(10,0)*+{(1,-5)}="h"; (40,0)*+{(3,-5)}="i";%
(10,-15)*+{\vdots}="j"; (25,-15)*+{\vdots}="k";(40,-15)*+{\vdots}="m";%
{\ar "a";"b"};{\ar "a";"c"};%
{\ar "b";"d"};{\ar "c";"d"};%
{\ar "d";"a"};{\ar "d";"e"};{\ar "d";"f"};%
{\ar "e";"b"};{\ar "e";"g"};{\ar "f";"c"};{\ar "f";"g"};%
{\ar "g";"d"};{\ar "g";"h"};{\ar "g";"i"};%
{\ar "h";"e"};{\ar "i";"f"};%
{\ar "j";"h"};{\ar "k";"g"};{\ar "m";"i"};%
{\ar "h";"k"};{\ar "i";"k"};%
\end{xy}}
\end{minipage}
\begin{minipage}[t]{.5\linewidth}
\centerline{
\begin{xy}
(15,60)*+{(2,0)}="1";%
(15,50)*+{(2,-2)}="2"; (30,50)*+{(1,-1)}="3";%
(0,40)*+{(1,-3)}="4"; (15,40)*+{(2,-4)}="5";%
(15,30)*+{(2,-6)}="6"; (30,30)*+{(1,-5)}="7";%
(0,20)*+{(1,-7)}="8"; (15,20)*+{(2,-8)}="9";%
(15,10)*+{(2,-10)}="10"; (30,10)*+{(1,-9)}="11";%
(0,0)*+{(1,-11)}="12";(15,0)*+{(2,-12)}="13";%
(0,-15)*+{\vdots}="14";(15,-15)*+{\vdots}="15";(30,-15)*+{\vdots}="16";%
{\ar "2";"1"};{\ar "5";"2"};{\ar "6";"5"};{\ar "9";"6"};{\ar "10";"9"};{\ar "13";"10"};{\ar "15";"13"};%
{\ar "12";"8"};{\ar "8";"4"};{\ar "11";"7"};{\ar "7";"3"};{\ar "16";"11"};{\ar "14";"12"};%
{\ar "2";"4"};{\ar "4";"6"};{\ar "6";"8"};{\ar "8";"10"};{\ar "10";"12"};{\ar "12";"15"};%
{\ar "1";"3"};{\ar "3";"5"};{\ar "5";"7"};{\ar "7";"9"};{\ar "9";"11"};{\ar "11";"13"};{\ar "13";"16"};%
\end{xy}}
\end{minipage}}
\caption{The quvier $G^-$ in type $\mathbb{A}_3$ (left) and the quvier $G^-$ in type $\mathbb{B}_2$ (right).} \label{initial quivers A3 and B2}
\end{figure}

Let ${\bf z}=\{z_{i,r} \mid (i,r)\in G^-_0\}$ and let $\mathscr{A}$ be the cluster algebra defined by the initial seed $({\bf z},G^-)$. The cluster algebra $\mathscr{A}$ is a cluster algebra of infinite rank.

Let ${\bf Y}^-=\{Y^{\pm 1}_{i,r}\mid (i,r)\in G^-_0\}$ be a new set of indeterminates over $\mathbb{Q}$. For $(i,r)\in G^-_0$, we define $k_{i,r}$ to be the unique positive integer $k$ satisfying 
\[
0 < kb_{ii}-|r| \leq b_{ii}.
\] 
In other words, $(i,r)$ is the $k$th vertex in its column, counting from the top.

For $(i,r)\in G^-_0$, we perform the substitution 
\begin{align}\label{variable substitution z-Y}
z_{i,r}= \prod_{j=0}^{k_{i,r}-1} Y_{i,r+jb_{ii}}.
\end{align}
Note that 
\[
\frac{z_{i,r}}{z_{i,r+b_{ii}}}=Y_{i,r}
\]
for $(i,r+b_{ii})\in G^-_0$.

Let $\Gamma$ be the same quiver as $G$ but with vertex set $\Gamma_0=\{(i,r-d_i):(i,r)\in G_0\}$. Let $\Gamma^-$ be the full subquiver of $\Gamma$ with vertex set $\Gamma^-_0=\Gamma_0\cap (I\times \mathbb{Z}_{\leq 0})$, see Figure \ref{initial gamma quivers A3 and B2}.

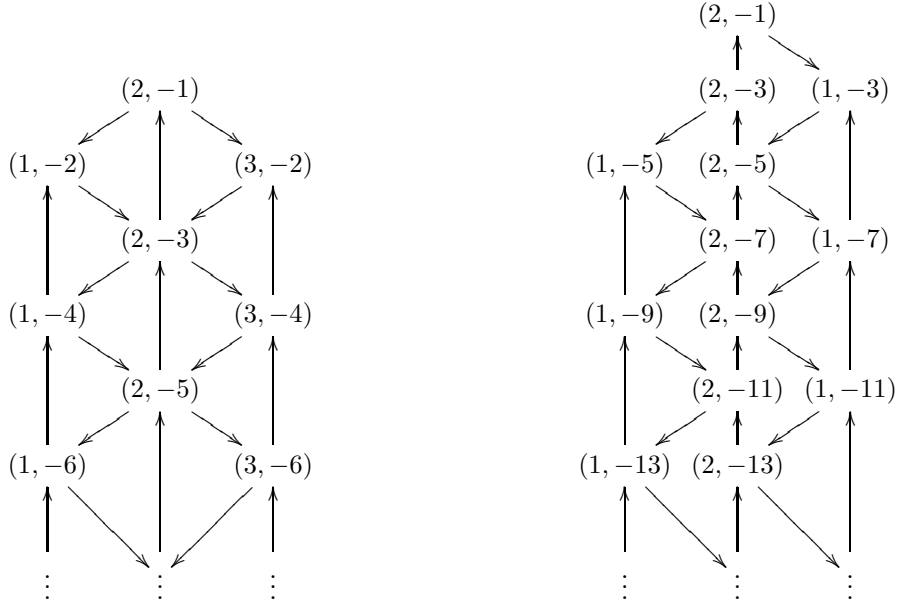
\begin{figure}
\resizebox{1.0\width}{1.0\height}{
\begin{minipage}[t]{.5\linewidth}
\centerline{
\begin{xy}
(25,50)*+{(2,-1)}="a";%
(10,40)*+{(1,-2)}="b"; (40,40)*+{(3,-2)}="c";%
(25,30)*+{(2,-3)}="d";%
(10,20)*+{(1,-4)}="e"; (40,20)*+{(3,-4)}="f";%
(25,10)*+{(2,-5)}="g";%
(10,0)*+{(1,-6)}="h"; (40,0)*+{(3,-6)}="i";%
(10,-15)*+{\vdots}="j"; (25,-15)*+{\vdots}="k";(40,-15)*+{\vdots}="m";%
{\ar "a";"b"};{\ar "a";"c"};%
{\ar "b";"d"};{\ar "c";"d"};%
{\ar "d";"a"};{\ar "d";"e"};{\ar "d";"f"};%
{\ar "e";"b"};{\ar "e";"g"};{\ar "f";"c"};{\ar "f";"g"};%
{\ar "g";"d"};{\ar "g";"h"};{\ar "g";"i"};%
{\ar "h";"e"};{\ar "i";"f"};%
{\ar "j";"h"};{\ar "k";"g"};{\ar "m";"i"};%
{\ar "h";"k"};{\ar "i";"k"};%
\end{xy}}
\end{minipage}
\begin{minipage}[t]{.5\linewidth}
\centerline{
\begin{xy}
(15,60)*+{(2,-1)}="1";%
(15,50)*+{(2,-3)}="2"; (30,50)*+{(1,-3)}="3";%
(0,40)*+{(1,-5)}="4"; (15,40)*+{(2,-5)}="5";%
(15,30)*+{(2,-7)}="6"; (30,30)*+{(1,-7)}="7";%
(0,20)*+{(1,-9)}="8"; (15,20)*+{(2,-9)}="9";%
(15,10)*+{(2,-11)}="10"; (30,10)*+{(1,-11)}="11";%
(0,0)*+{(1,-13)}="12";(15,0)*+{(2,-13)}="13";%
(0,-15)*+{\vdots}="14";(15,-15)*+{\vdots}="15";(30,-15)*+{\vdots}="16";%
{\ar "2";"1"};{\ar "5";"2"};{\ar "6";"5"};{\ar "9";"6"};{\ar "10";"9"};{\ar "13";"10"};{\ar "15";"13"};%
{\ar "12";"8"};{\ar "8";"4"};{\ar "11";"7"};{\ar "7";"3"};{\ar "16";"11"};{\ar "14";"12"};%
{\ar "2";"4"};{\ar "4";"6"};{\ar "6";"8"};{\ar "8";"10"};{\ar "10";"12"};{\ar "12";"15"};%
{\ar "1";"3"};{\ar "3";"5"};{\ar "5";"7"};{\ar "7";"9"};{\ar "9";"11"};{\ar "11";"13"};{\ar "13";"16"};%
\end{xy}}
\end{minipage}}
\caption{The quvier $\Gamma^-$ in type $\mathbb{A}_3$ (left) and the quvier $\Gamma^-$ in type $\mathbb{B}_2$ (right).} \label{initial gamma quivers A3 and B2} 
\end{figure}

In this paper, we let $\mathfrak{g}$ be of type $\mathbb{A}$ or $\mathbb{B}$. We work in the full subcategory $\mathscr{C}^-$ of $\mathscr{C}$ whose objects have all their composition factors of the form $L(m)$, where $m$ is a monomial in the variables $Y_{i,r}\in {\bf Y}^-$. 

\subsection{Quantum affine algebras}
Let $\mathfrak{g}$ be a simple complex Lie algebra whose Dynkin diagram has vertex set $I$ and $h^\vee$ be the dual Coxeter number of $\mathfrak{g}$, see Table \ref{dual Coxeter numbers}. Let $\mathfrak{\widehat{g}}$ be the corresponding untwisted affine Lie algebra which is realized as a central extension of the loop algebra $\mathfrak{g} \otimes\mathbb{C}[t,t^-]$. Let $U_q(\widehat{\mathfrak{g}})$ be the Drinfeld-Jimbo quantum enveloping algebra (quantum affine algebra for short) of $\mathfrak{\widehat{g}}$ with parameter $q\in \mathbb{C}^*$ not a root of unity, see \cite{CP94}. 

\begin{table}[H]
\begin{tabular}{c|ccccccccc}
\hline
$\mathfrak{g}$ & $\mathbb{A}_n$ & $\mathbb{B}_n$ &  $\mathbb{C}_n$ & $\mathbb{D}_n$ & $\mathbb{E}_6$ & $\mathbb{E}_7$ & $\mathbb{E}_8$ & $\mathbb{F}_4$ & $\mathbb{G}_2$ \\
$t$ & 1 & 2 & 2 & 1 & 1 & 1 & 1 & 2 & 3 \\
$h^\vee$ & $n+1$ & $2n-1$ & $n+1$ & $2n-2$ & 12 & 18 & 30 & 9 & 4 \\ 
\hline
\end{tabular}
\caption{Dual Coxeter numbers}\label{dual Coxeter numbers}
\end{table}

Let $U_q(\mathfrak{g})$ be the quantum enveloping algebra. Recall that a $U_q(\mathfrak{g})$-module $V$ is of type 1 if it is a direct sum of its weight subspaces. A $U_q(\widehat{\mathfrak{g}})$-module $V$ is said to be of type 1 if the central element $c^{1/2}$ acts as the identity on $V$, and if $V$ is of type 1 as a module $U_q(\mathfrak{g})$. Let $\mathscr{C}$ be the category of finite-dimensional $U_q(\widehat{\mathfrak{g}})$-modules of type 1. Every finite-dimensional simple $U_q(\widehat{\mathfrak{g}})$-module can be obtained from a type 1 module by twisting with an automorphism of $U_q(\widehat{\mathfrak{g}})$, see \cite{CP94,CP95a}. 

Let $K_0(\mathscr{C})$ be the Grothendieck ring of $\mathscr{C}$. Let $\mathcal{P}$ be the free abelian multiplicative group of monomials in infinitely many formal variables $(Y_{i,a})_{i\in I; a\in \mathbb{C}^\times}$. The \textit{$q$-character} of an object $M$ in $\mathscr{C}$ is defined as an injective ring homomorphism $\chi_q$ from $K_0(\mathscr{C})$ to the ring $\mathbb{Z}\mathcal{P}=\mathbb{Z}[Y^{\pm}_{i,a}]_{i\in I; a\in \mathbb{C}^\times}$ of Laurent polynomial in infinitely many formal variables. 

In this paper, we will be concerned only with polynomials involving the subset of variables $Y_{i,aq^r}$, $a\in \mathbb{C}^\times$, $(i,r)\in G_0$. For simplicity of notation, we write $Y_{i,r}$ for $Y_{i,aq^r}$. 

A monomial in $\mathbb{Z}\mathcal{P}$ is called {\it dominant} (respectively, {\it anti-dominant}) if it does not contain a factor $Y^{-1}_{i,r}$  (respectively, $Y_{i,r}$) with $(i,r)\in G_0$. Following \cite{FR98}, for $(i,r)\in \Gamma_0$, define 
\begin{align}\label{root analogue}
v_{i,r} := A^{-1}_{i,r}=Y^{-1}_{i,r-d_i}Y^{-1}_{i,r+d_i} \prod_{j:c_{ji}=-1} Y_{j,r} \prod_{j:c_{ji}=-2} Y_{j,r-1} Y_{j,r+1} \prod_{j:c_{ji}=-3} Y_{j,r-2} Y_{j,r}  Y_{j,r+2},
\end{align}
where the $c_{ij}$ are the entries of the Cartan matrix. It follows that $A_{i,r}$ is a Laurent monomial in the variables $Y_{j,s}$ with $(j,s)\in G_0$, see Section 2.3.2 of \cite{HL13}.

For any simple object $V$ in $\mathscr{C}$, it was shown by Frenkel and Mukhin \cite{FM01} that the $q$-character can be expressed as 
\begin{align*}
\chi_q(V) = m_+(1+\sum_p M_p), 
\end{align*}
where $m_+\in \mathbb{Z}\mathcal{P}$ is a monomial in the variables $Y_{i,r}$, $(i,r)\in G_0$, with positive powers, hence $m_+$ is a dominant monomial, and each $M_p$ is a product of factors $A^{-1}_{i,r}, (i,r)\in \Gamma_0$. The monomial $m_+$ is called the \textit{highest weight monomial} of $V$. There is a partial order $\leq$ on $\mathcal{P}$ defined by
\begin{align*}
m \leq m' \text{ if and only if $m'm^{-1}$ is a monomial generated by $A_{i,r}, (i,r)\in \Gamma_0$}.
\end{align*}
Then $m_+$ is maximum with respect to $\leq$. 

Every simple object in $\mathscr{C}$ can be parametrized by the highest weight monomial occurring in its $q$-character \cite{CP91,FR98}. The highest weight monomial is dominant, but in general the highest weight monomial is not the only dominant monomial occurring in $q$-characters. Given a dominant monomial $m$, one can construct the corresponding simple module $L(m)$. 

A simple module $L(m)$ is called \textit{special} or \textit{minuscule} if $m$ is the only dominant monomial occurring in $\chi_q((L(m))$, see Definition 10.1 of \cite{Nak04} or Section 5.2.2 of \cite{HL10}. It is \textit{anti-special} if there is exactly one anti-dominant monomial occurring in its $q$-character. Clearly, a special or anti-special module must be simple. A simple module is called \textit{thin} if any weight space of the simple module has no dimension greater than 1.
 
Following \cite{HL10}, define the \textit{truncated $q$-character} $\chi^-_q(L(m))$ to be the Laurent polynomial obtained from $\chi_q(L(m))$ by deleting all the monomials involving variables $Y_{i,r}\not\in {\bf Y}^-$. In other words, $\chi^-_q(L(m))\in \mathbb{Z}[Y^{\pm}_{i,r}\mid (i,r)\in G^-_0]$. By Proposition 3.10 of \cite{HL16}, $\chi^-_q$ is an injective ring homomorphism from the Grothendieck ring of $\mathscr{C}^-$ to $\mathbb{Z}[Y^{\pm}_{i,r}\mid (i,r)\in G^-_0]$.

\subsection{Paths}\label{MY paths}
Define a subset $\mathcal{X} \subset I \times \mathbb{Z}$ and an injective mapping $\iota: \mathcal{X} \to \mathbb{Z}\times \mathbb{Z}$ as follows.
\begin{gather}
\begin{align*}
&\text{Type $\mathbb{A}_n$}: \text{ Let } \mathcal{X}:=\{(i,k)\in I \times \mathbb{Z}: i-k \equiv 0 {\hskip -0.7em}\pmod 2\} \text{ and } \iota(i,k)=(i,k).\\
&\text{Type $\mathbb{B}_n$}: \text{ Let } \mathcal{X}:=\{(n,2k): k\in \mathbb{Z} \}\sqcup \{(i,k)\in I \times \mathbb{Z}: i<n \text{ and } k \equiv 1{\hskip -0.7em} \pmod 2\} \text{ and }\\
&\qquad \qquad \quad \iota(i,k) =
\begin{cases}
     (2i, k),  & \text{if } i<n \text{ and } 2n+k-2i \equiv 1{\hskip -0.7em}\pmod 4,  \\
     (4n-2-2i,k), & \text{if } i<n \text{ and } 2n+k-2i \equiv 3{\hskip -0.7em} \pmod 4, \\
     (2n-1, k), & \text{if } i=n.
\end{cases}
\end{align*}
\end{gather}

Following \cite{MY12a,MY12b}, for every $(i,k)\in \mathcal{X}$, a set $\mathscr{P}_{i,k}$ of paths is defined as follows. Here a path is a finite sequence of points in the plane $\mathbb{R}^{2}$. We write $(j,\ell) \in p$ if $(j,\ell)$ is a point of the path $p$. In our diagrams, we connect consecutive points of a path by line segments for illustrative purposes only.

The following is the case of type $\mathbb{A}_n$. For all $(i,k)\in I\times \mathbb{Z}$, let
\begin{align*}
\mathscr{P}_{i,k}=\{ & ((0,y_{0}),(1,y_{1}),\ldots,(n+1,y_{n+1})):  y_{0}=i+k, \\
&y_{n+1}=n+1-i+k, \text{ and } y_{j+1}-y_{j}\in \{1,-1\}, \  0\leq j\leq n\}.
\end{align*}
In other words, a path in $\mathscr{P}_{i,k}$ must start at $(0,i+k)$ and end at $(n+1,n+1-i+k)$ and each step between them can either go up one unit or go down one unit. So $|\mathscr{P}_{i,k}|=\binom{n+1}{i}$.

Note that the cardinality of $\mathscr{P}_{i,k}$ is equal to the number of Young tableaux that fit in an $i\times (n+1-i)$ rectangle.

\begin{figure}
\resizebox{1.0\width}{1.0\height}{
\begin{minipage}[b]{0.5\linewidth}
\centerline{
\begin{tikzpicture}
\draw[step=.5cm,gray,thin] (-0.5,5.5) grid (2,8) (-0.5,5.5)--(2,5.5);
\draw[fill] (0,8.3) circle (2pt) -- (0.5,8.3) circle (2pt) --(1,8.3) circle (2pt) --(1.5,8.3) circle (2pt);
\begin{scope}[thick, every node/.style={sloped,allow upside down}]
\draw (-0.5,7)--node {\midarrow}(0,6.5);
\draw (0,6.5)--node {\midarrow}(0.5,6);
\draw (0.5,6)--node {\midarrow}(1,5.5);
\draw (1,5.5)--node {\midarrow}(1.5,6);
\draw (1.5,6)--node {\midarrow}(2,6.5);
\draw (-0.5,7)--node {\midarrow}(0,7.5);
\draw (0,7.5)--node {\midarrow}(0.5,8);
\draw (0.5,8)--node {\midarrow}(1,7.5);
\draw (1,7.5)--node {\midarrow}(1.5,7);
\draw (1.5,7)--node {\midarrow}(2,6.5);
\draw (0,7.5)--node {\midarrow}(0.5,7);
\draw (0.5,7)--node {\midarrow}(1,6.5);
\draw (1,6.5)--node {\midarrow}(1.5,6);
\draw (0,6.5)--node {\midarrow}(0.5,7);
\draw (0.5,7)--node {\midarrow}(1,7.5);
\draw (0.5,6)--node {\midarrow}(1,6.5);
\draw (1,6.5)--node {\midarrow}(1.5,7);
\end{scope}
\node [above] at (0,8.3) {$1$};
\node [above] at (0.5,8.3) {$2$};
\node [above] at (1,8.3) {$3$};
\node [above] at (1.5,8.3) {$4$};
\node [left] at (-0.5,8) {$0$};
\node [left] at (-0.5,7.5) {$1$};
\node [left] at (-0.5,7) {$2$};
\node [left] at (-0.5,6.5) {$3$};
\node [left] at (-0.5,6) {$4$};
\node [left] at (-0.5,5.5) {$5$};
\end{tikzpicture}}
\end{minipage}}
\caption{In type $\mathbb{A}_4$: illustration of the paths in $\mathscr{P}_{2,0}$.}\label{a4path}
\end{figure}
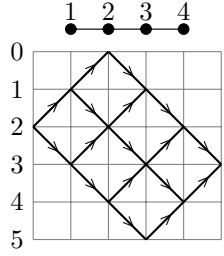

The sets $C_{p}^{\pm}$ of upper and lower corners of a path $p=((r,y_{r}))_{0\leq r \leq n+1}\in \mathscr{P}_{i,k}$ are defined as follows (see Figure \ref{a4path}):
\begin{align*}
C^{+}_{p}=\{(r,y_{r})\in p: r\in I, \ y_{r-1}=y_{r}+1=y_{r+1}\},\\
C^{-}_{p}=\{(r,y_{r})\in p: r\in I, \ y_{r-1}=y_{r}-1=y_{r+1}\}.
\end{align*}

The following is the case of type $\mathbb{B}_n$. Fix an $\varepsilon$ such that $0<\varepsilon <1/2$, for all $\ell \in 2\mathbb{Z}$, the set $\mathscr{P}_{n,\ell}$ is defined as follows.

For all $\ell \equiv 2 \pmod 4$,
\[
\begin{split}
\mathscr{P}_{n,\ell}= \{ & ((0,y_{0}), (2,y_{1}), \ldots, (2n-4,y_{n-2}), (2n-2,y_{n-1}),(2n-1,y_{n})): y_{0}=\ell+2n-1,\\
& y_{i+1}-y_{i}\in \{2,-2\}, \  0\leq i\leq n-2,\ \text{and}\ y_{n}-y_{n-1}\in \{1+\epsilon, -1-\epsilon\}\}.
\end{split}
\]

For all $\ell \equiv 0 \pmod 4$,
\[
\begin{split}
\mathscr{P}_{n,\ell}=& \{((4n-2,y_{0}),(4n-4,y_{1}),\ldots,(2n+2,y_{n-2}), (2n,y_{n-1}),(2n-1,y_{n})): y_{0}=\ell+2n-1,\\
& \ y_{i+1}-y_{i}\in \{2,-2\}, \   0\leq i\leq n-2, \ \text{and}\  y_{n}-y_{n-1}\in \{1+\epsilon, -1-\epsilon\}\}.
\end{split}
\]
In other words, a path in $\mathscr{P}_{n,\ell}$ must start at $(0,\ell+2n-1)$ or at $(4n-2,\ell+2n-1)$ and then it can either go up or go down until $(2n-1,y_{n})$. So $|\mathscr{P}_{n,\ell}|=2^n$.

For all $(i,k)\in \mathcal{X}$, $i< n$, let
\begin{align*}
\mathscr{P}_{i,k}=&\{(a_{0},a_{1},\ldots,a_{n},\overline{a}_{n},\ldots,\overline{a}_{1},\overline{a}_{0}):  (a_{0},a_{1},\ldots,a_n)\in \mathscr{P}_{n,k-(2n-2i-1)},\\
&(\overline{a}_{0},\overline{a}_{1},\ldots,\overline{a}_{n})\in \mathscr{P}_{n,k+(2n-2i-1)}, \text { and } a_{n}-\overline{a}_{n}= (0,y) \text{ where } y>0\}.
\end{align*}
In other words, a path in $\mathscr{P}_{i,k}$, $i<n$, must start at $(0,k+2i)$ (respectively, $(4n-2,k+2i)$) and end at $(4n-2,k+4n-2i-2)$ (respectively, $(0,k+4n-2i-2)$). It is not hard to get $|\mathscr{P}_{n,\ell}|>\binom{2n-1}{i}$.

The sets $C^{\pm}_{p}$ of upper and lower corners of a path $p=((j_{r},\ell_{r}))_{0\leq r \leq |p|-1}\in \mathscr{P}_{i,k}$,
where $|p|$ is the number of points in the path $p$, are defined as follows:
\begin{align*}
C^{+}_{p}=\ & \iota^{-1}\{(j_{r},\ell_{r})\in p: j_{r}\not \in \{0,2n-1,4n-2\}, \ \ell_{r-1}> \ell_{r},\ \ell_{r+1}>\ell_{r}\} \\
\ & \sqcup \{(n,\ell)\in \mathcal{X}: (2n-1,\ell-\epsilon)\in p \text{ and }(2n-1,\ell+\epsilon)\not\in p\}, \\
C^{-}_{p}=\ & \iota^{-1}\{(j_{r},\ell_{r})\in p: j_{r}\not \in \{0,2n-1,4n-2\}, \ \ell_{r-1}< \ell_{r},\ \ell_{r+1}<\ell_{r}\} \\
\ & \sqcup \{(n,\ell)\in \mathcal{X}: (2n-1,\ell-\epsilon)\not\in p \text{ and } (2n-1,\ell+\epsilon)\in p\}.
\end{align*}
These definitions are illustrated in Figures \ref{(4,2)} and \ref{(2,1)}.

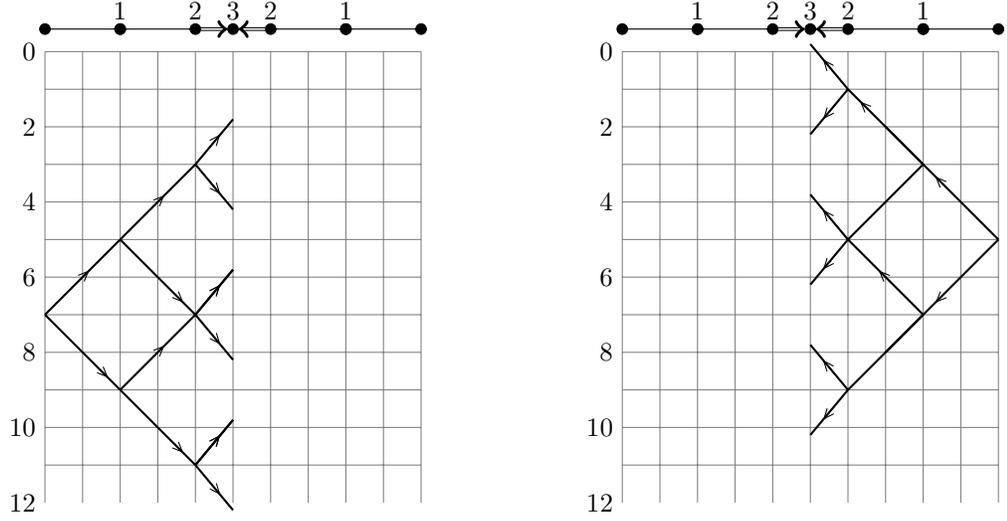
\begin{figure}
\resizebox{1.0\width}{1.0\height}{
\begin{minipage}[b]{0.5\linewidth}
\centerline{
\begin{tikzpicture}
\draw[step=.5cm,gray,thin] (0,2) grid (5,8);
\draw[fill] (0,8.3) circle (2pt)--(1,8.3) circle (2pt);
\draw[fill] (1,8.3) circle (2pt)--(2,8.3) circle (2pt);
\draw[fill] (3,8.3) circle (2pt)--(4,8.3) circle (2pt);
\draw[fill] (4,8.3) circle (2pt)--(5,8.3) circle (2pt);
\draw[fill] (2.5,8.3) circle (2pt);
\node [above] at (1,8.3) {1};
\node [above] at (2,8.3) {2};
\node [above] at (2.5,8.3) {3};
\node [above] at (3,8.3) {2};
\node [above] at (4,8.3) {1};
\draw [double,->] (2.075,8.3)--(2.425,8.3);
\draw [double,->] (2.925,8.3)--(2.575,8.3);
\node [left] at (0,8) {0};
\node [left] at (0,7) {2};
\node [left] at (0,6) {4};
\node [left] at (0,5) {6};
\node [left] at (0,4) {8};
\node [left] at (0,3) {10};
\node [left] at (0,2) {12};
\begin{scope}[thick, every node/.style={sloped,allow upside down}]
\draw (0,4.5)--node {\midarrow}(1,5.5);
\draw (1,5.5)--node {\midarrow}(2,6.5);
\draw (2,6.5)--node {\midarrow}(2.5,7.1);
\draw (2,6.5)--node {\midarrow}(2.5,5.9);
\draw (0,4.5)--node {\midarrow}(1.5,3);
\draw (1.5,3)--node {\midarrow}(2,2.5);
\draw (2,2.5)--node {\midarrow}(2.5,1.9);
\draw (2,2.5)--node {\midarrow}(2.5,3.1);
\draw (1,5.5)--(1.5,5) (1.5,5)--node {\midarrow}(2,4.5);
\draw (2,4.5)--node {\midarrow}(2.5,3.9);
\draw (2,4.5)--node {\midarrow}(2.5,5.1);
\draw (1,3.5)--node {\midarrow}(2,4.5);
\draw (2,4.5)--node {\midarrow}(2.5,5.1);
\draw (2,2.5)--node {\midarrow}(2.5,3.1);
\end{scope}
\end{tikzpicture}}
\end{minipage}
\begin{minipage}[b]{0.5\linewidth}
\centerline{
\begin{tikzpicture}
\draw[step=.5cm,gray,thin] (0,2) grid (5,8);
\draw[fill] (0,8.3) circle (2pt)--(1,8.3) circle (2pt);
\draw[fill] (1,8.3) circle (2pt)--(2,8.3) circle (2pt);
\draw[fill] (3,8.3) circle (2pt)--(4,8.3) circle (2pt);
\draw[fill] (4,8.3) circle (2pt)--(5,8.3) circle (2pt);
\draw[fill] (2.5,8.3) circle (2pt);
\node [above] at (1,8.3) {1};
\node [above] at (2,8.3) {2};
\node [above] at (2.5,8.3) {3};
\node [above] at (3,8.3) {2};
\node [above] at (4,8.3) {1};
\draw [double,->] (2.075,8.3)--(2.425,8.3);
\draw [double,->] (2.925,8.3)--(2.575,8.3);
\node [left] at (0,8) {0};
\node [left] at (0,7) {2};
\node [left] at (0,6) {4};
\node [left] at (0,5) {6};
\node [left] at (0,4) {8};
\node [left] at (0,3) {10};
\node [left] at (0,2) {12};
\begin{scope}[thick, every node/.style={sloped,allow upside down}]
\draw (5,5.5)--node {\midarrow}(3.5,7);
\draw (4,6.5)--(3.5,7) (3.5,7)--node {\midarrow}(3,7.5);
\draw (3,7.5)--node {\midarrow}(2.5,8.1);
\draw (3,7.5)--node {\midarrow}(2.5,6.9);
\draw (5,5.5)--node {\midarrow}(3.5,4);
\draw (4,4.5)--(3,3.5) (3,3.5)--node {\midarrow}(2.5,2.9);
\draw (3,3.5)--node {\midarrow}(2.5,4.1);
\draw (4,4.5)--node {\midarrow}(3,5.5);
\draw (4,6.5)--(3,5.5) (3,5.5)--node {\midarrow}(2.5,4.9);
\draw (3,5.5)--node {\midarrow}(2.5,6.1);
\end{scope}
\end{tikzpicture}}
\end{minipage}}
\caption{In type $\mathbb{B}_3$: left, $\mathscr{P}_{3,2}$; right, $\mathscr{P}_{3,0}$.}\label{(4,2)}
\end{figure}

\begin{figure}
\resizebox{1.0\width}{1.0\height}{
\begin{minipage}[b]{0.5\linewidth}
\centerline{
\begin{tikzpicture}
\draw[step=.5cm,gray,thin] (0,1.5) grid (5,8);
\draw[fill] (0,8.3) circle (2pt)--(1,8.3) circle (2pt);
\draw[fill] (1,8.3) circle (2pt)--(2,8.3) circle (2pt);
\draw[fill] (3,8.3) circle (2pt)--(4,8.3) circle (2pt);
\draw[fill] (4,8.3) circle (2pt)--(5,8.3) circle (2pt);
\draw[fill] (2.5,8.3) circle (2pt);
\node [above] at (1,8.3) {1};
\node [above] at (2,8.3) {2};
\node [above] at (2.5,8.3) {3};
\node [above] at (3,8.3) {2};
\node [above] at (4,8.3) {1};
\draw [double,->] (2.075,8.3)--(2.425,8.3);
\draw [double,->] (2.925,8.3)--(2.575,8.3);
\node [left] at (0,8) {0};
\node [left] at (0,7) {2};
\node [left] at (0,6) {4};
\node [left] at (0,5) {6};
\node [left] at (0,4) {8};
\node [left] at (0,3) {10};
\node [left] at (0,2) {12};
\begin{scope}[thick, every node/.style={sloped,allow upside down}]
\draw (0,6.5)--node {\midarrow}(1,7.5);
\draw (1,7.5)--node {\midarrow}(2,6.5);
\draw (2,6.5)--node {\midarrow}(2.5,5.9);
\draw (2.5,6.1)--node {\midarrow}(3,5.5);
\draw (3,5.5)--node {\midarrow}(4,4.5);
\draw (4,4.5)--node {\midarrow}(5,3.5);
\draw (0,6.5)--node {\midarrow}(1,5.5);
\draw (1,5.5)--node {\midarrow}(2,4.5);
\draw (2,4.5)--node {\midarrow}(2.5,3.9);
\draw (2.5,4.1)--node {\midarrow}(3,3.5);
\draw (3,3.5)--node {\midarrow}(4,2.5);
\draw (4,2.5)--node {\midarrow}(5,3.5);
\draw (1,5.5)--node {\midarrow}(2,6.5);
\draw (2,4.5)--node {\midarrow}(2.5,5.1) (2.5,4.9)--node {\midarrow}(3,5.5);
\draw (3,3.5)--node {\midarrow}(4,4.5);
\draw (2.5,3.9)--node {\midarrow}(2.5,5);
\draw (2.5,5)--node {\midarrow}(2.5,6.1);
\end{scope}
\end{tikzpicture}}
\end{minipage}
\begin{minipage}[b]{0.5\linewidth}
\centerline{
\begin{tikzpicture}
\draw[step=.5cm,gray,thin] (0,1.5) grid (5,8);
\draw[fill] (0,8.3) circle (2pt)--(1,8.3) circle (2pt);
\draw[fill] (1,8.3) circle (2pt)--(2,8.3) circle (2pt);
\draw[fill] (3,8.3) circle (2pt)--(4,8.3) circle (2pt);
\draw[fill] (4,8.3) circle (2pt)--(5,8.3) circle (2pt);
\draw[fill] (2.5,8.3) circle (2pt);
\node [above] at (1,8.3) {1};
\node [above] at (2,8.3) {2};
\node [above] at (2.5,8.3) {3};
\node [above] at (3,8.3) {2};
\node [above] at (4,8.3) {1};
\draw [double,->] (2.075,8.3)--(2.425,8.3);
\draw [double,->] (2.925,8.3)--(2.575,8.3);
\node [left] at (0,8) {0};
\node [left] at (0,7) {2};
\node [left] at (0,6) {4};
\node [left] at (0,5) {6};
\node [left] at (0,4) {8};
\node [left] at (0,3) {10};
\node [left] at (0,2) {12};
\begin{scope}[thick, every node/.style={sloped,allow upside down}]
\draw (5,5.5)--node {\midarrow}(4,6.5);
\draw (5,5.5)--node {\midarrow}(4,4.5);
\draw (4,4.5)--node {\midarrow}(3,3.5);
\draw (4,6.5)--node {\midarrow}(3,7.5);
\draw (4,6.5)--node {\midarrow}(3,5.5);
\draw (3,7.5)--node {\midarrow}(2.5,6.9);
\draw (3,5.5)--node {\midarrow}(2.5,4.9);
\draw (3,3.5)--node {\midarrow}(2.5,2.9);
\draw (2.5,7.1)--node {\midarrow}(2,6.5);
\draw (2.5,5.1)--node {\midarrow}(2,4.5);
\draw (2.5,3.1)--node {\midarrow}(2,2.5);
\draw (2,6.5)--node {\midarrow} (1,5.5);
\draw (1,5.5)--node {\midarrow} (0,4.5);
\draw (2,4.5)--node {\midarrow} (1,3.5);
\draw (2,2.5)--node {\midarrow} (1,3.5);
\draw (1,3.5)--node {\midarrow} (0,4.5);
\draw (2.5,2.9)--node {\midarrow}(2.5,4) (2.5,4)--node {\midarrow}(2.5,5) (2.5,5)--node {\midarrow}(2.5,6) (2.5,6)--node {\midarrow}(2.5,7.1);
\draw (2.5,5.9)--node {\midarrow} (2,6.5);
\draw (2,4.5)--node {\midarrow} (1,5.5);
\draw (2.5,3.9)--node {\midarrow} (2,4.5);
\draw (3,5.5)--node {\midarrow} (2.5,6.1);
\draw (4,4.5)--node {\midarrow} (3,5.5);
\draw (3,3.5)--node {\midarrow} (2.5,4.1);
\end{scope}
\end{tikzpicture}}
\end{minipage}}
\caption{In type $\mathbb{B}_3$: left, $\mathscr{P}_{1,1}$; right, $\mathscr{P}_{2,1}$.}\label{(2,1)}
\end{figure}
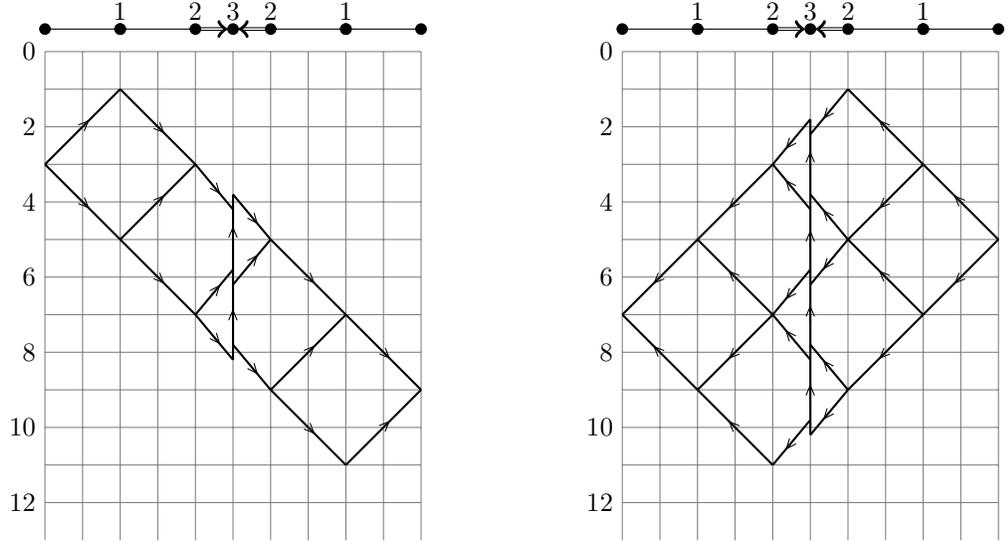

In order to subsequently describe our $F$-polynomials, it is helpful to define the notion of \textit{cell}. We call a region in $\mathscr{P}_{i,k}$ a \textit{cell} if it is a minimal region enclosed by paths. For every cell, we define the coordinate of the cell as follows. If the cell is a square or a square missing a corner, its coordinate is defined as the coordinate of the intersection of two diagonals. If the cell is a right triangle, its coordinate is defined as the coordinate of the midpoint of its hypotenuse. It is obvious that for any $(i,k)\in G_0$, our cell coordinate is an element of $\Gamma_0$.

We also need the following notations in this paper. For all $(i,k)\in \mathcal{X}$, let $p^{+}_{i,k}$ be the highest path which is the unique path in $\mathscr{P}_{i,k}$ with no lower corners and $p^{-}_{i,k}$ the lowest path which is the unique path in $\mathscr{P}_{i,k}$ with no upper corners. Let $p,p'$ be paths. We say that $p$ is \textit{strictly above} $p'$ or $p'$ is \textit{strictly below} $p$ if
\begin{align*}
(x,y)\in p \text{ and } (x,z)\in p' \Longrightarrow y < z.
\end{align*}

\subsection{Snake modules}
A simple module $L(m)$ is called a Kirillov-Reshetikhin module if $m$ is of the form
\begin{align}\label{KR expression}
m=\prod_{j=0}^{k-1}Y_{i,r+jb_{ii}}, \quad (i\in I, r\in \mathbb{Z}, k\geq 1),
\end{align}
and is usually denoted by $W^{(i)}_{k,r}$.

For completeness we recall the definition of snake module introduced by Mukhin and Young in~\cite{MY12a,MY12b}. Let $(i,k) \in \mathcal{X}$. A point $(i',k')\in \mathcal{X}$ is said to be in \textit{snake position} with respect to $(i,k)$ if
\begin{gather}
\begin{align*}
&\text{Type $\mathbb{A}_n$}: \; k'-k \geq |i'-i|+2 \ \text{and} \  k'-k \equiv |i'-i|{\hskip -0.7em} \pmod 2.  \\
&\text{Type $\mathbb{B}_n$}: \; i=i'=n:\ k'-k \geq  2 \ \text{and} \ k'-k \equiv  2 {\hskip -0.7em}\pmod 4,  \\
&\qquad \qquad \quad i \neq i'=n \text{ or } i'\neq i=n: \ k'-k \geq 2|i'-i|+3  \ \text{and} \  k'-k \equiv  2|i'-i|-1 {\hskip -0.7em}\pmod 4, \\
&\qquad \qquad \quad i < n \text{ and }i' < n:\ k'-k \geq 2|i'-i|+4 \  \text{and} \ k'-k \equiv  2|i'-i|{\hskip -0.7em} \pmod 4.
\end{align*}
\end{gather}
The point $(i',k')$ is in \textit{minimal} snake position to $(i,k)$ if $k'-k$ is equal to the given lower bound. The point $(i',k')$ is in \textit{prime} snake position to $(i,k)$ if 
\begin{gather}
\begin{align*}
&\text{Type $\mathbb{A}_n$}: \; \min \{ 2n+2-i-i', i+i' \} \geq k'-k \geq |i'-i|+2 \ \text{and} \  k'-k \equiv |i'-i|{\hskip -0.7em} \pmod 2.  \\
&\text{Type $\mathbb{B}_n$}:\; i=i'=n:\ 4n-2 \geq  k'-k \geq  2 \ \text{and} \ k'-k \equiv  2 {\hskip -0.7em}\pmod 4,  \\
&{\hskip 5em} i \neq i'=n \text{ or } i'\neq i=n: \  2i'+2i-1 \geq  k'-k \geq 2|i'-i|+3  \ \text{and} \  k'-k \equiv  2|i'-i|-1 {\hskip -0.7em}\pmod 4, \\
&{\hskip 5em} i < n \text{ and }i' < n: \  2i'+2i \geq  k'-k \geq 2|i'-i|+4 \  \text{and} \ k'-k \equiv  2|i'-i| {\hskip -0.7em}\pmod 4.
\end{align*}
\end{gather}

A finite sequence $(i_{t},k_{t})$, $1 \leq t \leq T$, $T \in \mathbb{Z}_{\geq 0}$, of points in $\mathcal{X}$ is called a \textit{snake} (respectively, \textit{prime snake}, \textit{minimal snake}) if for all $2 \leq t \leq T$, the point $(i_{t},k_{t})$ is in snake position (respectively, prime snake position, minimal snake position)  with respect to $(i_{t-1},k_{t-1})$ \cite{MY12a, MY12b}. 

The simple module $L(m)$ is called a \textit{snake module} (respectively, \textit{prime snake module}, \textit{minimal snake module}) if $m=\prod_{t=1}^{T}Y_{i_{t},k_{t}}$ for some snake (respectively, prime snake, minimal snake) $(i_{t},k_{t})_{1 \leq t \leq T}$ \cite{MY12a, MY12b}. In this case, we say that $(i_{t},k_{t})_{1 \leq t \leq T}$ is the snake of $L(m)$.

\begin{theorem}{\cite[Section 4.1]{DLL19}}\label{snake modules theorem}
The snake modules of type $\mathbb{A}_n$ or $\mathbb{B}_n$ are precisely the $U_q(\widehat{\mathfrak{g}})$-modules $L(m)$ with highest weight monomial
\begin{align}\label{the dominant monomial of snake modules}
m=\prod_{j=1}^{N} \left(\prod_{s=0}^{k_{j}-1} Y_{i_j,r+b_{i_{j}i_{j}}s+\sum_{\ell=1}^{j-1}n_{\ell}} \right),
\end{align}
where $r\in\mathbb{Z}$, $i_{j}\in I,\ k_{j} \geq 1$ for $1 \leq j \leq N$; furthermore $n_{\ell} = b_{i_\ell i_\ell}(k_\ell-1) + 2t+t|i_{\ell+1}-i_{\ell}|+2tj_\ell + \varepsilon_{i_\ell,i_{\ell+1}}$, where $j_\ell \in \mathbb{Z}_{\geq 0}$ for $1 \leq \ell \leq N-1$, 
\begin{align*}
\varepsilon_{i,j} = \begin{cases}
0 & \text{$\mathfrak{g}$ is of type $\mathbb{A}_n$,} \\
-\delta_{in}-\delta_{jn} & \text{$\mathfrak{g}$ is of type $\mathbb{B}_n$,}
\end{cases}  \quad 
t = \begin{cases}
1 & \text{$\mathfrak{g}$ is of type $\mathbb{A}_n$,} \\
2 & \text{$\mathfrak{g}$ is of type $\mathbb{B}_n$,}
\end{cases} 
\end{align*}
where $\delta_{ij}$ is the Kronecker delta, and we use the convention $\sum_{\ell=1}^{0}n_{\ell}=0$. In particular $L(m)$ is a prime snake module if $j_\ell$ satisfies the following bounded conditions:
\begin{gather}
\begin{align*}
& \text{Type $\mathbb{A}_n$}: 
0 \leq j_\ell \leq \begin{cases} 
n-\max\{i_\ell,i_{\ell+1}\} & \text{ if } i_\ell+i_{\ell+1}\geq n+1; \\
\min\{i_\ell,i_{\ell+1}\}-1 & \text{ if } i_\ell+i_{\ell+1}<n+1.
\end{cases}      \\
& \text{Type $\mathbb{B}_n$}: \;  0\leq j_\ell \leq \min\{i_\ell,i_{\ell+1}\}-1.
\end{align*}
\end{gather}
\end{theorem}

\begin{example}
\begin{itemize}
	\item[(1)] Every Kirillov-Reshetikhin module is a snake module, by taking $N=1$ in Theorem \ref{snake modules theorem}.
	\item[(2)] Minimal affinizations introduced in \cite{C95} (see also \cite{ZDLL16}) are snake modules, by taking $i_1<i_2<\ldots<i_N$ or $i_1>i_2>\ldots>i_N$ and $j_\ell=0$ in Theorem \ref{snake modules theorem}.
	\item[(3)] For examples of snake modules which are not Kirillov-Reshetikhin modules or minimal affinizations see Examples \ref{exampleA31}--\ref{exampleB21}.  
\end{itemize}
\end{example}

A graphic interpretation is that in $G^-$, the upper bound of the integer $j_\ell$ is the minimum distance from the columns with vertex labelings $(i_\ell,-)$ and $(i_{\ell+1},-)$ to the leftmost column and the rightmost column respectively. The value of $j_\ell$ is associated to the multiplicity of the left and right operators defined in Section 5.5 of \cite{DLL19}.

Let $(i_{\ell},k_{\ell}) \in \mathcal{X}$, $1 \leq \ell \leq T$, be a snake of length $T \in \mathbb{Z}_{\geq 1}$ and $p_\ell\in\mathscr{P}_{i_{\ell},k_{\ell}}$. We say that a $T$-tuple of paths $(p_{1},\ldots,p_{T})$ is \textit{non-overlapping} if $p_{s}$ is strictly above $p_{t}$ for all $1\leq s<t\leq T$. Let
\begin{align*}\label{non-overlapping T-path}
\overline{\mathscr{P}}_{(i_{t},k_{t})_{1\leq t\leq T}}=\{(p_{1},\ldots,p_{T}): p_{t}\in \mathscr{P}_{i_{t},k_{t}}, \ 1\leq t\leq T, \ (p_{1},\ldots,p_{T})\text { is } \text {non-overlapping} \}.
\end{align*}

Mukhin and Young have proved the following theorem.
\begin{theorem}[{\cite[Theorem 6.1]{MY12a}; \cite[Theorem 6.5]{MY12b}}] \label{path description of q-characters}
Let $(i_{\ell},k_{\ell}) \in \mathcal{X}$, $1 \leq \ell \leq T$, be a snake of length $T \in \mathbb{Z}_{\geq 1}$. Then
\begin{align}
\chi_{q} (L(\prod_{\ell=1}^{T} Y_{i_{\ell}, k_{\ell}})) = \sum_{(p_{1},\ldots,p_{T}) \in \overline{\mathscr{P}}_{(i_{\ell},k_{\ell})_{1 \leq \ell \leq T}}} \prod_{\ell=1}^{T}\mathfrak{m}(p_{\ell}),
\end{align}
where the mapping $\mathfrak{m}$ is defined by
\begin{align*}
\mathfrak{m}: \bigsqcup_{(i,k)\in \mathcal{X}}{\hskip -0.5em}\mathscr{P}_{i,k} & \longrightarrow \mathbb{Z}[Y^{\pm}_{j,\ell}]_{(j,\ell)\in \mathcal{X}} \nonumber \\
p& \longmapsto  \mathfrak{m}(p)=\prod_{(j,\ell)\in C^{+}_{p}}{\hskip -0.5em}Y_{j,\ell}{\hskip -0.5em}\prod_{(j,\ell)\in C^{-}_{p}}{\hskip -0.5em}Y^{-1}_{j,\ell}.
\end{align*} 
Moreover, the module $L(\prod_{\ell=1}^{T} Y_{i_{\ell}, k_{\ell}})$ is thin, special and anti-special.
\end{theorem}

\begin{remark}\label{overlapping property of snake modules}
A snake module $L(\prod_{\ell=1}^{T} Y_{i_{\ell}, k_{\ell}})$ is prime if and only if for all $2\leq t\leq T$ the paths $p^+_{i_t,k_t}\in \mathscr{P}_{i_t,k_t}$ and $p^-_{i_{t-1},k_{t-1}}\in \mathscr{P}_{i_{t-1},k_{t-1}}$ are overlapping.
\end{remark}

In view of Theorem \ref{path description of q-characters}, the $q$-characters of snake modules of types $\mathbb{A}_n$ and $\mathbb{B}_n$ with length $T$ are given by a set of $T$-tuples of non-overlapping paths, the path in each $T$-tuple is non-overlapping. This property is called the \textit{non-overlapping property}.

For any two paths $p_1,p_2\in \mathscr{P}_{i,k}$, $p_1$ can be obtained from $p_2$ by a sequence of moves, see Lemma~5.8 of \cite{MY12a}. We say that $p_1\leq p_2$ if $\mathfrak{m}(p_1)\leq \mathfrak{m}(p_2)$.

In the following we list some known facts about snake modules.
\begin{theorem}{\cite[Proposition 3.1]{MY12b}} \label{factorization of snake modules}
A snake module is prime if and only if its snake is prime. Every snake module can be uniquely written as a tensor product of prime snake modules (up to permutation). 
\end{theorem}

\begin{theorem}{\cite[Theorem 3.4,Theorem 5.9]{DLL19}}\label{prime sm cluster variables}
Prime snake modules are real and they correspond to some cluster variables in the cluster algebra $\mathscr{A}$. 
\end{theorem}

Moreover, in Theorem 4.1 of \cite{MY12b}, Mukhin and Young introduced a set of 3-term recurrence relations satisfied by $q$-characters of prime snake modules, called extended $T$-system, which generalizes the usual $T$-system. Moreover, in Theorem 4.1 of \cite{DLL19}, the authors introduced a system of equations satisfied by $q$-characters of prime snake modules, called $S$-system, which contains the usual $T$-system. In fact the equations in the $S$-system can be interpreted as cluster transformations in the cluster algebra $\mathscr{A}$ where the initial cluster variables correspond to certain Kirillov-Reshetikhin modules.

\subsection{Quivers with potentials}
Following \cite{HL16}, for every $i\neq j$ with $c_{ij}\neq 0$, and every $(i,r)\in \Gamma^-_0$, we have in $\Gamma^-$ an oriented cycle with length $2+|c_{ij}|$:
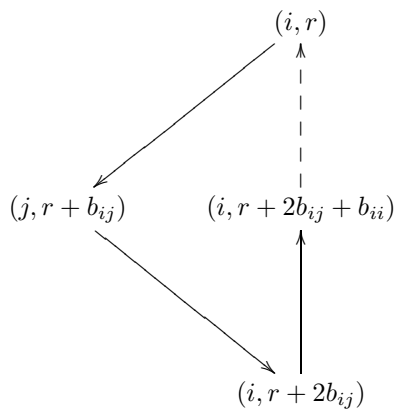
\begin{align*}
\xymatrix{
& (i,r) \ar[ddl] \\
&\\
(j,r+b_{ij})\ar[ddr] & (i,r+2b_{ij}+b_{ii}) \ar@{-->}[uu] \\
& \\
& (i,r+2b_{ij}) \ar[uu]}
\end{align*}
A potential $S$ is defined as the formal (infinite) sum for all these oriented cycles up to cyclic permutations, see Section 3 of \cite{DWZ08}. Hence in $\Gamma^-$, all the cyclic derivatives of $S$, introduced in Definition 3.1 of \cite{DWZ08}, are finite sums of paths. Indeed, a given arrow of $\Gamma^-$ can only occur in a finite number of summands. 

Let $R$ be the set of all cyclic derivatives of $S$. Let $J$ be the two-sided ideal of the path algebra $\mathbb{C}\Gamma^-$ generated by $R$. Following \cite{DWZ10,HL16}, one defines the Jacobian algebra $A=\mathbb{C}\Gamma^-/J$. Then $A$ is an infinite-dimensional $\mathbb{C}$-algebra.

Let $M$ be a finite-dimensional $A$-module, and $e\in \mathbb{N}^{\Gamma^-_0}$ be a dimension vector. Let $\text{Gr}_e(M)$ be the quiver Grassmannian of $M$. Thus $\text{Gr}_e(M)$ is the variety of submodules of $M$ with dimension vector $e$. This is a projective complex variety. Denote by $\chi(\text{Gr}_e(M))$ its Euler characteristic. Following~\cite{DWZ10,HL16}, define the $F$-polynomial of $M$ as a polynomial in the indeterminates $v_{i,r}$, $(i,r)\in \Gamma^-_0$, as follows:
\begin{align*}
F_M=\sum_{e\in \mathbb{N}^{\Gamma^-_0}} \chi(\text{Gr}_e(M)) \prod_{(i,r)\in \Gamma^-_0} v^{e_{i,r}}_{i,r}.
\end{align*}
It was shown in \cite{DWZ10} that for any finite-dimensional $M$, $F_M$ is a monic polynomial with constant term equal to $1$.

Following Section 4.5.2 of \cite{HL16}, let $\ell\in \mathbb{Z}_{<0}$ and let $\Gamma^-_\ell$ be the full subquiver of $\Gamma^-$ with vertex set 
\[
(\Gamma^-_0)_\ell:=\{(i,m)\in \Gamma^-_0\mid m\geq \ell\}.
\]  
Let $S_\ell$ be the sum of all cycles in the potential $S$ which only involve vertices of $(\Gamma^-_0)_\ell$, called a truncation of $S$. Let $J_\ell$ be the two-sided ideal of $\mathbb{C}\Gamma^-_\ell$ generated by all cyclic derivatives of $S_\ell$ and let
\[
A_\ell = \mathbb{C}\Gamma^-_\ell/J_\ell
\]
be the truncated Jacobian algebra at height $\ell$. Denote by $\pi:\mathbb{C}\Gamma^-_\ell \to A_\ell$ the natural projection.

It has been shown in Proposition 4.17 of \cite{HL16} that for any $\ell$, $A_\ell$ is finite-dimensional and the quiver with potential $(\Gamma^-_\ell,J_\ell)$ is rigid, namely, every cycle is cyclically equivalent to an element of $J_\ell$.  

\subsection{$q$-characters and $F$-polynomials} 
Let $m$ be a dominant monomial in the variables $Y_{i,r}$ for $(i,r)\in G^-_0$. Following \cite{FZ07,HL16}, for each $(i,r)\in G^-_0$, define
\begin{align*}
\widehat{y}_{i,r} = \prod_{(i,r)\to (j,s)} z_{j,s} \prod_{(j,s)\to(i,r)} z^{-1}_{j,s}.
\end{align*}
It was shown in Lemma 4.15 of \cite{HL16} that $\widehat{y}_{i,r}=A^{-1}_{i,r-d_i}$ for $(i,r)\in G^-_0$, so $\widehat{y}_{i,r}$ is a monomial in the variables $Y_{i,s}$, $(i,s)\in G^-_0$ by (\ref{root analogue}).

Using \cite[Corollary 6.3]{FZ07}, Hernandez and Leclerc gave the following formula for a cluster variable in terms of its $F$-polynomial and ${\bf g}$-vector. Every cluster variable $x$ of $\mathscr{A}$ has the following form
\begin{align}\label{cluster variable F and g}
x={\bf z}^{g_x} F_x({\bf \widehat{y}}).
\end{align} 

On the other hand, in \cite{FM01}, the truncated $q$-character $\chi^-_q(L(m))$ is expressed as 
\begin{align}\label{m-Pm}
\chi^-_q(L(m))= m P_m,
\end{align}
where $P_m$ is a polynomial with integer coefficients in the variables $\{A^{-1}_{i,r-d_i}\mid (i,r)\in G^-_0\}$ and has constant term $1$. Thus, by \cite{HL16}, if $L(m)$ is a cluster variable of $\mathscr{A}$, then $m={\bf z}^{g(m)}$, where the integer vector $g(m)\in \mathbb{Z}^{G^-_0}$ is the ${\bf g}$-vector of $L(m)$.

Let $I_{i,r}$ be the indecomposable injective $A$-module at vertex $(i,r)\in \Gamma^-_0$. Motivated by quivers with potentials \cite{DWZ10} and cluster character \cite{Pal08,Pal12}, Hernandez and Leclerc defined the following notion of generic kernel. 

\begin{definition}{\cite[Definition 4.5 and Section 5.2.2]{HL16}} \label{definition of generic kernel}
Let $K(m)$ be the kernel of a generic $A$-module homomorphism from the injective $A$-module $I(m)^-$ to the injective $A$-module $I(m)^+$, where 
\[
I(m)^+=\bigoplus_{g_{i,r}(m)>0} I_{i,r-d_i}^{\oplus g_{i,r}(m)}, \quad  I(m)^-=\bigoplus_{g_{i,r}(m)<0} I_{i,r-d_i}^{\oplus |g_{i,r}(m)|}.
\]
\end{definition}

The support of $K(m)$ is the collection of all points $(j,s)\in \Gamma^-_0$ such that the $(j,s)$-component of $K(m)$ is nonzero. We denote by $\text{Supp}(K(m))$ the support of $K(m)$. 

In \cite{HL16}, Hernandez and Leclerc proposed the following conjecture.
\begin{conjecture}{\cite[Conjecture 5.3]{HL16}}\label{conjecture5.3}
Let $L(m)$ be a real simple $U_q(\widehat{\mathfrak{g}})$-mdoule in $\mathscr{C}^-$. Then up to normalization, the truncated $q$-character of $L(m)$ is equal to the $F$-polynomial of the associated generic kernel. More precisely, 
\begin{align*}
\chi^-_q(L(m)) = mF_{K(m)},
\end{align*}
where the variables $v_{i,r}$ of the $F$-polynomial are evaluated as in (\ref{root analogue}). 
\end{conjecture}

In Theorem 4.8 of \cite{HL16}, Hernandez and Leclerc proved Conjecture \ref{conjecture5.3} for Kirillov-Reshetikhin modules, that is, up to renormalizing, the truncated $q$-character of the Kirillov-Reshetikhin module $W^{(i)}_{k,r-d_i(2k-1)}$ is equal to the $F$-polynomial of the generic kernel $F_{K^{(i)}_{k,r}}$, where $K^{(i)}_{k,r}$ is the kernel of a generic $A$-module homomorphism from $I_{i,r}$ to $I_{i,r-kb_{ii}}$. We will prove Conjecture \ref{conjecture5.3} for snake modules in Theorem \ref{theorem of geometric formula for sm}.

\subsection{A formula for the lowest weight monomial}
Recall that $t=\max\{d_i\mid i\in I\}$ as defined in Section \ref{section2.1}. As a generalization of Remark 4.14 of \cite{HL16}, we can calculate the dimension vectors of the $A$-module $K^{(i)}_{k,r}$ for $r\leq d_i(2k-1)-th^\lor$. Indeed, by Lemma 6.8 and Corollary 6.9 of \cite{FM01}, the lowest monomial of $\chi_q(\prod_{s=1}^k Y_{i,r-d_i(2s-1)})$ is equal to $\prod_{s=1}^k Y^{-1}_{v(i),r-d_i(2s-1)+th^\lor}$, where $v$ is the involution of $I$ defined by $w_0(\alpha_i)=-\alpha_{v(i)}$, where $w_0$ is the longest element in the Weyl group of $\mathfrak{g}$. Using Theorem 4.8 of \cite{HL16}, we can calculate the lowest monomial, which corresponds to the term in the $F$-polynomial for the trivial submodule $K^{(i)}_{k,r}\subset K^{(i)}_{k,r}$. Thus
\begin{align*}\label{dimension vector}
\prod_{s=1}^k Y^{-1}_{v(i),r-d_i(2s-1)+th^\lor}=\left(\prod_{s=1}^k Y_{i,r-d_i(2s-1)}\right) \prod_{(j,s)\in \Gamma^-_0} v^{d_{j,s}(K^{(i)}_{k,r})}_{j,s},
\end{align*}
where $(d_{j,s}(K^{(i)}_{k,r}))_{(j,s)\in \mathbb{N}^{\Gamma^-_0}}$ is the dimension vector of $K^{(i)}_{k,r}$.

In the next section, we will introduce a combinatorial method to calculate the dimension vector of the $A$-module $K(m)$ associated to the snake module $L(m)$.   

\section{A geometric character formula for snake modules}\label{GFSM}
In this section, we show that the geometric character formula conjectured by Hernandez and Leclerc holds for snake modules of types $\mathbb{A}$ and $\mathbb{B}$. We give a combinatorial formula for the $F$-polynomial of the generic kernel $K(m)$ associated to the snake module $L(m)$. As a consequence, we obtain a combinatorial method to compute the dimension vector of $K(m)$ as well as all its submodules.

\subsection{A geometric character formula for snake modules}

We first give a description of the ${\bf g}$-vector $g(m):=(g_{i,s})_{(i,s)\in G^-_0}$ for arbitrary prime snake module $L(m)$.

\begin{proposition}\label{g-vectors}
Let $L(m)$ be a prime snake module with highest weight monomial $m$ of the form~(\ref{the dominant monomial of snake modules}). Then we can rewrite
\[
m = {\bf z}^{g(m)} := \prod_{(i,s)\in G^-_0} z^{g_{i,s}(m)}_{i,s}, 
\]
where 
\begin{align*}
g_{i,s}(m)= \begin{cases}
1  & \text{if } (i,s)=(i_j,r+\sum_{\ell=1}^{j-1}n_{\ell}) \text{ and } r+\sum_{\ell=1}^{j-1}n_{\ell}\leq 0, \\
-1 & \text{if } (i,s)=(i_j,r+\sum_{\ell=1}^{j-1}n_{\ell}+b_{i_{j}i_{j}}k_j)  \text{ and }r+\sum_{\ell=1}^{j-1}n_{\ell}+b_{i_{j}i_{j}}k_j\leq 0, \\
0  & \text{otherwise}.
\end{cases}
\end{align*}
Here $j=1,\cdots,N$ as in (\ref{the dominant monomial of snake modules}). 
\end{proposition}

\begin{proof}
From Theorem \ref{snake modules theorem}, it follows that every prime snake module $L(m)$ is a $U_q(\widehat{\mathfrak{g}})$-module with highest weight monomial $m$ of the form (\ref{the dominant monomial of snake modules}).  Thus $m$ is a product of terms of the form 
\begin{align*} 
\prod_{s=0}^{k_{j}-1} Y_{i_j,r+\sum_{\ell=1}^{j-1}n_{\ell}+b_{i_{j}i_{j}}s}.
\end{align*}

Because of (\ref{KR expression}), for any $1\leq j\leq N$,
\[
L\left(\prod_{s=0}^{k_{j}-1} Y_{i_j,r+\sum_{\ell=1}^{j-1}n_{\ell}+b_{i_{j}i_{j}}s}\right)
\]
is a Kirillov-Reshitikhin module. Now the result follows from Theorem \ref{prime sm cluster variables} and Proposition 4.16 of~\cite{HL16}. 
\end{proof}

We are now ready for the main result of this section. The following theorem gives a positive answer to the Hernandez-Leclerc Conjecture (Conjecture \ref{conjecture5.3}) for snake modules.
\begin{theorem}\label{theorem of geometric formula for sm}
Let $L(m)$ be a prime snake module in $\mathscr{C}^-$. Then up to normalization,  the truncated $q$-character of $L(m)$ is equal to the $F$-polynomial of the associated generic kernel $K(m)$. More precisely, 
\begin{align*}
\chi^-_q(L(m)) = m F_{K(m)},
\end{align*}
where $F_{K(m)}$ is a polynomial in the variables (\ref{root analogue}). 
\end{theorem}
\begin{proof}
Recall that $\mathscr{A}$ is the cluster algebra defined in Section \ref{section2.1}. We use the characterization of $m$ from Theorem \ref{snake modules theorem}. The fact that $L(m) \in \mathscr{C}^-$ implies that each index of $Y$ in the formula (\ref{the dominant monomial of snake modules}) is a vertex in $G^-_0$. This implies that for some integer $N$, 
\begin{align}\label{star}
(i_N,r+\sum_{\ell=1}^{N-1}n_{\ell}+b_{i_Ni_N}(k_N-1))\in G^-_0. 
\end{align}
In particular, the second coordinate of (\ref{star}) is non-positive. Thus $r+\sum_{\ell=1}^{j-1}n_{\ell}+b_{i_ji_j}(k_j-1)\leq 0$ for all $j=1,\cdots,N$.

By Theorem \ref{prime sm cluster variables}, the truncated $q$-character $\chi^-_q(L(m))$ is a cluster variable $x$ of $\mathscr{A}$. By Proposition~\ref{g-vectors}, the ${\bf g}$-vector of $x$ is given by 
\begin{align*}
g_{i,s}(m)= \begin{cases}
1  & \text{if } (i,s)=(i_j,r+\sum_{\ell=1}^{j-1}n_{\ell}), \\
-1 & \text{if } (i,s)=(i_j,r+\sum_{\ell=1}^{j-1}n_{\ell}+b_{i_{j}i_{j}}k_j) \text{ and }r+\sum_{\ell=1}^{j-1}n_{\ell}+b_{i_{j}i_{j}}k_j\leq 0, \\
0  & \text{otherwise},
\end{cases}
\end{align*} 
where we use that $r+\sum_{\ell=1}^{j-1}n_{\ell}\leq 0$, because of (\ref{star}).

For $\ell<0$, let $(G^-_0)_\ell:=\{(i,r+d_i):(i,r)\in (\Gamma^-_0)_\ell\}$ and ${\bf z}^-_\ell=\{z_{i,r}\mid (i,r)\in (G^-_0)_\ell\}$. We denote by $G^-_\ell$ the same quiver as $\Gamma^-_\ell$, but with vertices labeled by $(G^-_0)_\ell$. Clearly, the cluster variable $x$ is a Laurent polynomial in the variables of ${\bf z}^-_\ell$ for some $\ell\ll 0$, and can be regarded as a cluster variable of the cluster algebra $\mathscr{A}_\ell$ defined by the initial seed $({\bf z}^-_\ell,G^-_\ell)$. 

The rest of the proof is similar to the proof of Theorem 4.8 in \cite{HL16}. Since the quiver with potential $(\Gamma^-_\ell,J_\ell)$ is rigid, we can apply the theory of \cite{DWZ08,DWZ10} and deduce that the $F$-polynomial of $x$ coincides with the polynomial $F_M$ associated with a certain $A_\ell$-module $M$. Futhermore $M$ is rigid by \cite{FK10,Am09}.
 
By Remark 4.1 of \cite{P12}, $M$ is the kernel of a generic element of the homomorphism space between two injective $A_\ell$-modules corresponding to the negative and positive components of the ${\bf g}$-vector of $x$. More precisely, let $I^\ell_{i,m}$ be the injective $A_\ell$-module at vertex $(i,m)$, then $M$ is the kernel of a generic element of $\text{Hom}(I^\ell(m)^-,I^\ell(m)^+)$, where
\[
I^\ell(m)^+= \bigoplus_{g_{i,s}(m)>0} {I^\ell_{i,s-d_i}}^{\oplus g_{i,s}(m)}, \quad  I^\ell(m)^-=\bigoplus_{g_{i,s}(m)<0} {I^\ell_{i,s-d_i}}^{\oplus |g_{i,s}(m)|}.
\]

It was shown in \cite{HL16} that our $A_\ell$-module $M$ does not change when $\ell$ increases and that in the direct limit 
\[
A=\lim_{\ell \to -\infty} A_\ell.
\]
The $A$-module $M$ is the kernel of a generic element of $\text{Hom}(I(m)^-,I(m)^+)$. Thus $M=K(m)$.
\end{proof}

From the proof of Theorem \ref{theorem of geometric formula for sm} we obtain the following corollay.
\begin{corollary}\label{Km rigid}
Let $L(m)$ be a prime snake module in $\mathscr{C}^-$. Then the generic kernel $K(m)$ is rigid and indecomposable.
\end{corollary}

\begin{remark}\label{snake modules geometric character formula remark}
By Theorem \ref{factorization of snake modules}, every snake module of type $\mathbb{A}_n$ or type $\mathbb{B}_n$ is isomorphic to a tensor product of prime snake modules defined uniquely up to permutation. On the other hand, if $M$ and $N$ are two finite-dimensional $A$-modules, then by Proposition 3.2 of \cite{DWZ10} we have $F_{M\oplus N}=F_MF_N$. Therefore, replacing the module $K(m)$ in Theorem \ref{theorem of geometric formula for sm} by a direct sum, we obtain a similar geometric character formula for arbitrary snake module of types $\mathbb{A}_n$ and $\mathbb{B}_n$.
\end{remark}
 
We present several examples to illustrate Theorem \ref{theorem of geometric formula for sm}.
\begin{example}\label{exampleA31}
In type $\mathbb{A}_3$, let $N=3$, $k_1=1,k_2=2,k_3=1$, $i_1=1, i_2=3, i_3=2$, $r=-15$, $j_1=j_2=0$, $n_1=4$, and $n_2=5$. Then $m=Y_{1,-15} Y_{3,-11} Y_{3,-9} Y_{2,-6}$. We get

\begin{align*}
g_{i,s}(m)= \begin{cases}
1  & \text{if } (i,s)=(1,-15), (3,-11), \text{ or } (2,-6); \\
-1 & \text{if } (i,s)=(1,-13), (3,-7), \text{ or } (2,-4); \\
0  & \text{otherwise}.
\end{cases}
\end{align*}
Thus by Definition \ref{definition of generic kernel}
\begin{align*}
I(m)^+= I_{1,-16}  \oplus I_{3,-12}\oplus I_{2,-7}, \quad  I(m)^-= I_{1,-14}  \oplus I_{3,-8} \oplus I_{2,-5}.
\end{align*}

The module $K(m)$ has dimension 13 and is displayed in Figure \ref{example1}. In Figure \ref{example1}, all vertices carry a vector space of dimension $1$. Applying Theorem \ref{theorem of geometric formula for sm}, we can compute its $q$-character as follows. There are 160 submodules in $K(m)$. 

\begin{align*}
\chi^-_q(L(m)) &= m ((1+v_{1,-14} + v_{1,-14}v_{2,-13})(v_{3,-8}+v_{3,-10}v_{3,-8} + v_{2,-5}v_{2,-7}v_{3,-8}\\
& +v_{2,-5}v_{2,-7}v_{3,-8}v_{3,-10}+ v_{1,-4}v_{1,-6} v_{2,-5}v_{2,-7}v_{3,-8}  \\
& + v_{1,-4}v_{1,-6} v_{2,-5}v_{2,-7}v_{3,-8}v_{3,-10}+ v_{2,-5}v_{2,-7}v_{2,-9}v_{3,-8}v_{3,-10}\\
& + v_{1,-4}v_{1,-6} v_{2,-5}v_{2,-7}v_{2,-9}v_{3,-8}v_{3,-10} \\
& + v_{1,-4}v_{1,-6}v_{1,-8} v_{2,-5}v_{2,-7}v_{2,-9}v_{3,-8}v_{3,-10}) \\
& + v_{1,-14}v_{2,-13}v_{3,-12}v_{3,-10}v_{3,-8} (v_{2,-5}v_{2,-7} + v_{2,-5}v_{2,-7}v_{2,-9}  \\
& + v_{1,-4}v_{1,-6} v_{2,-5}v_{2,-7}+v_{1,-4}v_{1,-6} v_{2,-5}v_{2,-7}v_{2,-9} \\
& + v_{1,-4}v_{1,-6}v_{1,-8} v_{2,-5}v_{2,-7}v_{2,-9}) \\
& + (1+v_{1,-14} + v_{1,-14}v_{2,-13} + v_{1,-14}v_{2,-13}v_{3,-12}v_{3,-10}v_{3,-8})(1+v_{2,-5} \\
& + v_{1,-4} v_{2,-5}+v_{3,-4} v_{2,-5} + v_{1,-4} v_{3,-4} v_{2,-5} + v_{1,-4} v_{3,-4} v_{2,-3} v_{2,-5}) \\ 
& + (1+v_{1,-14} + v_{1,-14}v_{2,-13})((v_{3,-8}+v_{3,-10}v_{3,-8})(v_{2,-5} + v_{1,-4} v_{2,-5} \\
& + v_{3,-4} v_{2,-5}+v_{1,-4} v_{3,-4} v_{2,-5} + v_{1,-4} v_{3,-4} v_{2,-3} v_{2,-5}) \\
& +(v_{2,-5}v_{2,-7}v_{3,-8} +v_{2,-5}v_{2,-7}v_{3,-8}v_{3,-10}+ v_{2,-5}v_{2,-7}v_{2,-9}v_{3,-8}v_{3,-10})(v_{1,-4} \\
& + v_{3,-4} +v_{1,-4}v_{3,-4} +v_{2,-3}v_{1,-4}v_{3,-4})\\
& + (v_{3,-4}+v_{3,-4}v_{2,-3})(v_{1,-4}v_{1,-6} v_{2,-5}v_{2,-7}v_{3,-8} \\
& + v_{1,-4}v_{1,-6} v_{2,-5}v_{2,-7}v_{3,-8}v_{3,-10} \\
& + v_{1,-4}v_{1,-6} v_{2,-5}v_{2,-7}v_{2,-9}v_{3,-8}v_{3,-10} \\
& + v_{1,-4}v_{1,-6}v_{1,-8} v_{2,-5}v_{2,-7}v_{2,-9}v_{3,-8}v_{3,-10})) \\
& + v_{1,-14}v_{2,-13}v_{3,-12}v_{3,-10}v_{3,-8}(v_{2,-5}v_{2,-7} + v_{2,-5}v_{2,-7}v_{2,-9})(v_{1,-4}+v_{3,-4} \\
& + v_{1,-4} v_{3,-4} + v_{1,-4} v_{3,-4} v_{2,-3}) \\
& +v_{1,-14}v_{2,-13}v_{3,-12}v_{3,-10}v_{3,-8} (v_{3,-4}+v_{3,-4} v_{2,-3})  (v_{1,-4}v_{1,-6} v_{2,-5}v_{2,-7} \\
& + v_{1,-4}v_{1,-6} v_{2,-5}v_{2,-7}v_{2,-9}+v_{1,-4}v_{1,-6}v_{1,-8} v_{2,-5}v_{2,-7}v_{2,-9})).
\end{align*}

Starting from the initial seed $({\bf z}, G^-)$, the following sequence of mutations produces (in the last step) the cluster variable corresponding to $L(m)$.
\begin{align*}
& (3,-3), (2,-2), (1,-3), (2,-4), (1,-5), (2,-6), (3,-7), \\
& (1,-7), (2,-8), (3,-9), (3,-11),(2,-12),(1,-13).    
\end{align*} 

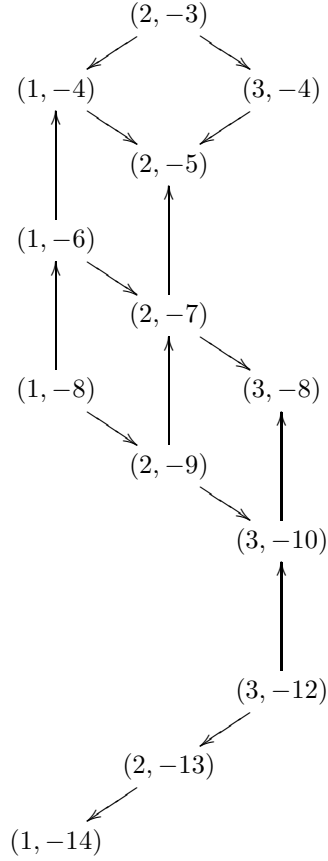
\begin{figure}
\resizebox{1.0\width}{1.0\height}{
\begin{minipage}[t]{.5\linewidth}
\centerline{
\begin{xy}
(10,100) *+{(1,-4)} ="0",
(10,80) *+{(1,-6)} ="1",
(10,60) *+{(1,-8)} ="2",
(10,0) *+{(1,-14)} ="5",
(25,110) *+{(2,-3)} ="8",
(25,90) *+{(2,-5)} ="9",
(25,70) *+{(2,-7)} ="10",
(25,50) *+{(2,-9)} ="11",
(25,10) *+{(2,-13)} ="13",
(40,100) *+{(3,-4)} ="16",
(40,60) *+{(3,-8)} ="18",
(40,40) *+{(3,-10)} ="19",
(40,20) *+{(3,-12)} ="20",
"1", {\ar"0"},
"8", {\ar"0"},
"0", {\ar"9"},
"2", {\ar"1"},
"1", {\ar"10"},
"2", {\ar"11"},
"13", {\ar"5"},
"8", {\ar"16"},
"10", {\ar"9"},
"16", {\ar"9"},
"11", {\ar"10"},
"10", {\ar"18"},
"11", {\ar"19"},
"20", {\ar"13"},
"19", {\ar"18"},
"20", {\ar"19"},
\end{xy}}
\end{minipage}}
\caption{The support of the $A$-module $K(Y_{1,-15} Y_{3,-11} Y_{3,-9} Y_{2,-6})$ in type $\mathbb{A}_3$.}\label{example1}
\end{figure}
\end{example}

\begin{example}\label{exampleA32}
In type $\mathbb{A}_3$, let $N=2$, $k_1=1,k_2=1$, $i_1=2,i_2=2$, $r=-10$, $j_1=1$, and $n_1=4$. Then $m=Y_{2,-10} Y_{2,-6}$. We get

\begin{align*}
g_{i,s}(m)= \begin{cases}
1  & \text{if } (i,s)=(2,-10), \text{ or } (2,-6); \\
-1 & \text{if } (i,s)=(2,-8), \text{ or } (2,-4); \\
0  & \text{otherwise}.
\end{cases}
\end{align*}
Thus by Definition \ref{definition of generic kernel}
\[
I(m)^+= I_{2,-11}\oplus I_{2,-7}, \quad I(m)^-= I_{2,-9}\oplus I_{2,-5}.
\]

The module $K(m)$ has dimension 8 and is displayed in Figure \ref{example2}. In Figure \ref{example2}, all vertices carry a vector space of dimension $1$. Applying Theorem \ref{theorem of geometric formula for sm}, we can compute its $q$-character as follows. There are 35 submodules in $K(m)$.

\begin{align*}
\chi^-_q(L(m)) & = m (1+ v_{2,-9}+v_{2,-9}v_{1,-8}+v_{2,-9}v_{3,-8}+v_{2,-9}v_{1,-8}v_{3,-8} \\
& +v_{2,-9}v_{1,-8}v_{3,-8}v_{2,-7}v_{2,-5} + (v_{2,-5} + v_{1,-4} v_{2,-5}+ v_{3,-4} v_{2,-5}  \\
& + v_{1,-4} v_{3,-4} v_{2,-5} + v_{1,-4} v_{3,-4} v_{2,-3} v_{2,-5})(1+ v_{2,-9} \\
& + v_{2,-9}v_{1,-8}+v_{2,-9}v_{3,-8}+v_{2,-9}v_{1,-8}v_{3,-8}) \\
& + v_{2,-9}v_{1,-8}v_{3,-8}v_{2,-7}v_{2,-5} (v_{1,-4}+ v_{3,-4} + v_{1,-4} v_{3,-4} + v_{1,-4} v_{3,-4} v_{2,-3})).
\end{align*}

Starting from the initial seed $({\bf z}, G^-)$, the following sequence of mutations produces (in the last step) the cluster variable corresponding to $L(m)$.
\begin{align*}
& (3,-3), (2,-2), (1,-3), (2,-4), (3,-7), (2,-6), (1,-7), (2,-8).    
\end{align*} 

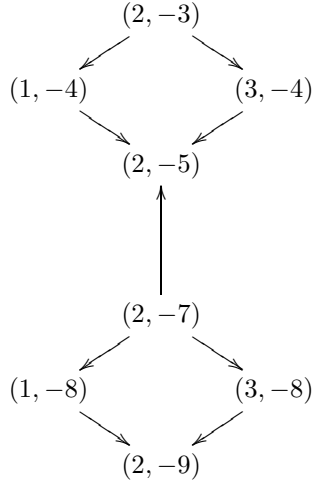
\begin{figure}
\resizebox{1.0\width}{1.0\height}{
\begin{minipage}[t]{.3\linewidth}
\centerline{
\begin{xy}
(10,100) *+{(1,-4)} ="0",
(10,60) *+{(1,-8)} ="2",
(25,110) *+{(2,-3)} ="8",
(25,90) *+{(2,-5)} ="9",
(25,70) *+{(2,-7)} ="10",
(25,50) *+{(2,-9)} ="11",
(40,100) *+{(3,-4)} ="16",
(40,60) *+{(3,-8)} ="18",
"8", {\ar"0"},
"0", {\ar"9"},
"2", {\ar"11"},
"8", {\ar"16"},
"18", {\ar"11"},
"16", {\ar"9"},
"10", {\ar"9"},
"10", {\ar"2"},
"10", {\ar"18"},
\end{xy}}
\end{minipage}}
\caption{The support of the $A$-module $K(Y_{2,-10} Y_{2,-6})$ in type $\mathbb{A}_3$.}\label{example2}
\end{figure}
\end{example}

\begin{example}\label{exampleB21}
In type $\mathbb{B}_2$, let $N=2$, $k_1=1,k_2=1$, $i_1=2,i_2=2$, $r=-12$, $j_1=1$, and $n_1=6$. Then $m=Y_{2,-12} Y_{2,-6}$. We get

\begin{align*}
g_{i,s}(m)= \begin{cases}
1  & \text{if } (i,s)=(2,-12), \text{ or } (2,-6); \\
-1 & \text{if } (i,s)=(2,-10), \text{ or } (2,-4); \\
0  & \text{otherwise}.
\end{cases}
\end{align*}
Thus by Definition \ref{definition of generic kernel}
\[
I(m)^+= I_{2,-13}  \oplus I_{2,-7}, \quad I(m)^-= I_{2,-11}  \oplus I_{2,-5}.
\]

The module $K(m)$ has dimension 6 and is displayed in Figure \ref{example3}. In Figure \ref{example3}, all vertices carry a vector space of dimension $1$. Applying Theorem \ref{theorem of geometric formula for sm}, we can compute its $q$-character as follows. There are 15 submodules in $K(m)$.

\begin{align*}
\chi^-_q(L(m)) = & m (1+ v_{2,-11} + v_{2,-11}v_{1,-9} + v_{2,-11}v_{1,-9}v_{2,-7}v_{2,-5} \\
& + (1+ v_{2,-11} + v_{2,-11}v_{1,-9}) (v_{2,-5} + v_{2,-5} v_{1,-3} + v_{2,-5} v_{1,-3} v_{2,-1}) \\
& + v_{2,-11}v_{1,-9}v_{2,-7}v_{2,-5} v_{1,-3} + v_{2,-11}v_{1,-9}v_{2,-7}v_{2,-5} v_{1,-3}v_{2,-1}).
\end{align*}

Starting from the initial seed $({\bf z}, G^-)$, the following sequence of mutations produces (in the last step) the cluster variable corresponding to $L(m)$.
\begin{align*}
(2,0), (1,-1), (2,-4), (2,-6), (1,-7), (2,-10).
\end{align*} 

\begin{figure}[H]
\resizebox{1.0\width}{1.0\height}{
\begin{minipage}[t]{.3\linewidth}
\begin{xy}
(15,60)*+{(2,-1)}="1";%
(30,50)*+{(1,-3)}="3";%
(15,40)*+{(2,-5)}="5";%
(15,30)*+{(2,-7)}="6"; 
(0,20)*+{(1,-9)}="8"; 
(15,10)*+{(2,-11)}="10"; 
{\ar "6"; "5"};
{\ar "6";"8"};
{\ar "8";"10"};
{\ar "1";"3"};
{\ar "3";"5"};
\end{xy}
\end{minipage}}
\caption{The support of the $A$-module $K(Y_{2,-12} Y_{2,-6})$ in type $\mathbb{B}_2$.}\label{example3}
\end{figure}
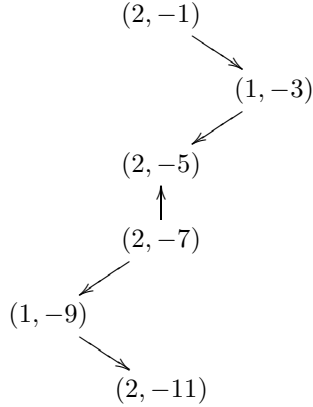
\end{example}


\subsection{Combinatorial character formula}
Recall from Section 2.3 that for $(i,k)\in G^-_0$, we denote by $p^+_{i,k}$ (respectively, $p^-_{i,k}$) the unique highest (respectively, lowest) path in $\mathscr{P}_{i,k}$. For an arbitrary path $p\in \mathscr{P}_{i,k}$, we let $p \ominus p^{+}_{i,k}$ denote the symmetric difference between $p$ and $p^{+}_{i,k}$, defined as $p \ominus p^{+}_{i,k}=(p \cup p^{+}_{i,k}) \backslash (p \cap p^{+}_{i,k})$. Then the set $p \ominus p^{+}_{i,k}$ encloses the union of some consecutive cells. 

\begin{remark}
In type $\mathbb{A}_n$, the mapping $p \mapsto p \ominus p^{+}_{i,k}$ defines a bijection between $\mathscr{P}_{i,k}$ and the set of all Young diagrams inside an $i\times (n+1-i)$ rectangle.
\end{remark}

We define the height monomial $h(p)$ of a path $p\in \mathscr{P}_{i,k}$ by
\[
h(p) = \prod_{(i,r)\in p \ominus p^{+}_{i,k}} v_{i,r},
\]
where $(i,r)\in \Gamma^-_0$ runs over all the cell coordinates in $p \ominus p^{+}_{i,k}$ and we use the convention: $v_{i,r}=0$ if $(i,r)\not\in \Gamma^-_0$. In particular, $h(p^{+}_{i,k})=1$.   

Recall that for any snake $(i_{t},k_{t})$, $1\leq t\leq T\in \mathbb{Z}_{\geq1}$,
\begin{align*} 
\overline{\mathscr{P}}_{(i_{t},k_{t})_{1\leq t\leq T}}=\{(p_{1},\ldots,p_{T}): p_{t}\in \mathscr{P}_{i_{t},k_{t}}, \ 1\leq t\leq T, \ (p_{1},\ldots,p_{T})\text { is } \text {non-overlapping} \}.
\end{align*}

\begin{theorem}\label{F-polynomial formula}
Let $L(m)=L(\prod_{i=1}^{T}Y_{i_t,k_t})$ be a prime snake module and $K(m)$ be the generic kernel associated to $L(m)$. Then 
\begin{align*}
F_{K(m)} = \sum_{(p_1,\ldots,p_T)\in \overline{\mathscr{P}}_{(i_t,k_t)_{1\leq t\leq T}}} \prod_{t=1}^T h(p_t).
\end{align*}
\end{theorem}
\begin{proof}
By Theorem \ref{theorem of geometric formula for sm}, we have 
\begin{align}\label{proof of F-polynomial formula-1}
F_{K(m)}=\frac{\chi^-_q(L(m))}{m},
\end{align}
and Theorem \ref{path description of q-characters} gives a formula for $\chi_q(L(m))$ in terms of paths
\begin{align}\label{proof of F-polynomial formula-2}
\chi_q(L(m))= \sum_{(p_{1},\ldots,p_{T}) \in \overline{\mathscr{P}}_{(i_{\ell},k_{\ell})_{1 \leq \ell \leq T}}} \prod_{\ell=1}^{T}\mathfrak{m}(p_{\ell}).
\end{align}

Note that equation (\ref{proof of F-polynomial formula-1}) uses the truncated $q$-character $\chi^-_q(L(m))$ whereas equation (\ref{proof of F-polynomial formula-2}) uses the complete $q$-character $\chi_q(L(m))$. First we prove the statement in the case where $\chi^-_q(L(m))=\chi_q(L(m))$.

Applying Lemma 5.10 of \cite{MY12a} and using induction, we have 
\begin{align}\label{proof of F-polynomial formula-3}
\prod_{t=1}^T \mathfrak{m}(p_t) = \prod_{t=1}^T \mathfrak{m}(p^+_{i_t,k_t}) \prod_{r=1}^R A^{-1}_{j_r,\ell_r}. 
\end{align}
where $(j_r,\ell_r)$, $1\leq r\leq R\in \mathbb{Z}_{\geq 0}$, is a sequence of cell coordinates determined by the symmetric difference $p^+_{i_t,k_t}\ominus p_t$, $1\leq t\leq T$. Therefore 
\begin{align}\label{proof of F-polynomial formula-4}
\prod_{t=1}^T h(p_t) = \prod_{r=1}^R A^{-1}_{j_r,\ell_r}. 
\end{align}
Moreover, since $p^+_{i_t,k_t}$ is the highest path in $\mathscr{P}_{i_t,k_t}$, Theorem  \ref{path description of q-characters} implies that $\prod_{t=1}^T \mathfrak{m}(p^+_{i_t,k_t})$ contains no negative powers. Since $L(m)$ is special, its highest weight monomial
$m$ is the unique dominant monomial in $\chi_q(L(m))$, and thus
\begin{align}\label{proof of F-polynomial formula-5}
m = \prod_{t=1}^T \mathfrak{m}(p^+_{i_t,k_t}). 
\end{align}
Thus equations (\ref{proof of F-polynomial formula-1})--(\ref{proof of F-polynomial formula-5}) imply 
\begin{align*}
mF_{K(m)} = \chi_q(L(m)) = \sum_{(p_{1},\ldots,p_{T}) \in \overline{\mathscr{P}}_{(i_{\ell},k_{\ell})_{1 \leq \ell \leq T}}} m \prod_{r=1}^R A^{-1}_{j_r,\ell_r}=\sum_{(p_1,\ldots,p_T)\in \overline{\mathscr{P}}_{(i_t,k_t)_{1\leq t\leq T}}} m \prod_{t=1}^T h(p_t).
\end{align*}

Now suppose $\chi^-_q(L(m))\neq\chi_q(L(m))$. Then we have to modify the above argument as follows. Equation (\ref{proof of F-polynomial formula-2}) is replaced by 
\begin{align}\label{proof of F-polynomial formula-6}
\chi^-_q(L(m))= \sum_{\substack{(p_{1},\ldots,p_{T}) \in \overline{\mathscr{P}}_{(i_{\ell},k_{\ell})_{1 \leq \ell \leq T}}\\ C^+_{p_\ell}, C^-_{p_\ell}\subset G^-_0}} \prod_{\ell=1}^{T}\mathfrak{m}(p_{\ell}).
\end{align}
In other words, we require that for each path $p_\ell$ the upper and lower corners $C^+_{p_\ell}, C^-_{p_\ell}$ lie in $G^-_0$. Moreover, in equation (\ref{proof of F-polynomial formula-4}), we replace $A^{-1}_{j_r,\ell_r}$ by ${A'}^{-1}_{j_r,\ell_r}$ where
\begin{align*}
{A'}^{-1}_{j_r,\ell_r} = \begin{cases}
A^{-1}_{j_r,\ell_r} & \text{ if $(j_r,\ell_r)\in \Gamma^-_0$}, \\
0 & \text{otherwise}.
\end{cases}
\end{align*}
Then 
\begin{align*}
mF_{K(m)} = \chi^-_q(L(m)) & = \sum_{\substack{(p_{1},\ldots,p_{T}) \in \overline{\mathscr{P}}_{(i_{\ell},k_{\ell})_{1 \leq \ell \leq T}}\\ C^+_{p_\ell}, C^-_{p_\ell}\subset G^-_0}} m \prod_{r=1}^{R'} A^{-1}_{j_r,\ell_r} \\
& =\sum_{(p_1,\ldots,p_T) \in \overline{\mathscr{P}}_{(i_t,k_t)_{1\leq t\leq T}}} m \prod_{r=1}^R {A'}^{-1}_{j_r,\ell_r} \\
& = \sum_{(p_1,\ldots,p_T)\in \overline{\mathscr{P}}_{(i_t,k_t)_{1\leq t\leq T}}} m \prod_{t=1}^T h(p_t).
\end{align*}
\end{proof}

\begin{remark}\label{compute dimensional vector}
\begin{itemize}
	\item[(1)] Theorem \ref{F-polynomial formula} allows us to calculate the dimension vector $(d_{i,r}(K(m)))_{(i,r)\in \Gamma^-_0}$ of the $A$-module $K(m)$ in a combinatorial way using all $T$-tuples of non-overlapping paths. We will explain this in the next section.
	\item[(2)] Theorem \ref{F-polynomial formula} provides a combinatorial approach to find all submodules of $K(m)$, see Examples \ref{example111}-\ref{example333}.
	\item[(3)] Using Proposition 3.2 of \cite{DWZ10}, for any two finite-dimensional $A$-modules $M$ and $N$, we have 
\[
F_{M\oplus N} = F_M F_N.
\]
Replacing the $A$-module $K(m)$ in Theorem \ref{F-polynomial formula} by a direct sum of such modules, we obtain a similar combinatorial formula for arbitrary snake modules.
\end{itemize}
\end{remark}

\begin{corollary}\label{corollary of Grassmannian}
If $L(m)=L(\prod_{i=1}^{T}Y_{i_t,k_t})$ is a snake module and $K(m)$ is the associated generic kernel, then for all dimension verctors $\underline{e}$ we have
\[
\chi(\text{Gr}_{\underline{e}}(K(m)))=0 \text{ or } 1.
\]
\end{corollary}
\begin{proof}
Using Theorem \ref{F-polynomial formula} and the definition of the $F$-polynomial, it suffices to show that for any two $T$-tuples $(p_1,\ldots,p_T)\neq (p'_1,\ldots,p'_T) \in \overline{\mathscr{P}}_{(i_t,k_t)_{1\leq t\leq T}}$ of non-overlapping paths, we have $\prod_{t=1}^T h(p_t) \neq \prod_{t=1}^T h(p'_t)$. This holds because $(p_1,\ldots,p_T)$ are disjoint paths and each $p_i$ is determined by $p^+_i\ominus p_i$.
\end{proof}

\begin{remark}
Corollary \ref{corollary of Grassmannian} holds for any thin and real module if the Conjecture 13.2 of \cite{HL10} or Conjecture 5.2 of \cite{HL13} or Conjecture 9.1 of \cite{Le10} holds. 
\end{remark}

\subsection{Generic kernel}\label{generic kernel}

Recall that $\mathscr{P}_{(i,k)}$ is a collection of paths defined in Section \ref{MY paths}. Let 
\[
\mathscr{P}'_{(i,k)}=\{p \in \mathscr{P}_{(i,k)}\mid C^\pm_{p}\subset G^-_0\} \subset \mathscr{P}_{(i,k)}.
\]

Let $\mathscr{P}_{(i_t,k_t)_{1\leq t \leq T}}$ be a collection of paths associated to a snake module $L(m)$ of the form (\ref{the dominant monomial of snake modules}) in $\mathscr{C}^-$. For any snake $(i_{t},k_{t})$, $1\leq t\leq T\in \mathbb{Z}_{\geq1}$, let
\begin{align*} 
\overline{\mathscr{P}'}_{(i_{t},k_{t})_{1\leq t\leq T}}=\{(p'_{1},\ldots,p'_{T}): p'_{t}\in \mathscr{P}'_{i_{t},k_{t}}, \ 1\leq t\leq T, \ (p'_{1},\ldots,p'_{T})\text { is } \text {non-overlapping} \}.
\end{align*}

Let $V_{m}$ be the set of all the cell coordinates in the set $\bigcup_{1\leq t\leq T} ({p'}^-_{i_t,k_t} \ominus p^{+}_{i_t,k_t})$, where ${p'}^-_{i_t,k_t}$ is a minimal path in $\mathscr{P}'_{(i_{t},k_{t})}$ for each $1\leq t\leq T$ and $({p'}^-_{i_1,k_1},\ldots,{p'}^-_{i_T,k_T})\in \overline{\mathscr{P}'}_{(i_{t},k_{t})_{1\leq t\leq T}}$.

Note that when $\chi^-_q(L(\prod_{i=1}^{T}Y_{i_t,k_t}))=\chi_q(L(\prod_{i=1}^{T}Y_{i_t,k_t}))$, the set $V_{m}$ is the set of all the cell coordinates in the set $\bigcup_{1\leq t\leq T} (p^-_{i_t,k_t}\ominus p^{+}_{i_t,k_t})$.

\begin{definition}
Let $Q(m)$ be the full subquiver of $\Gamma^-$ with vertex set $V_{m}$.
\end{definition}

If we assign a vector space whose dimension is equal to the multiplicity of cells with coordinate $(i,r)$ occurring in the multiset $\bigcup_{1\leq t\leq T} ({p'}^-_{i_t,k_t} \ominus p^{+}_{i_t,k_t})$ to every point $(i,r)\in V_m$, then by Theorem~\ref{F-polynomial formula}, the generic kernel $K(m)$ is a representation of $Q(m)$. In general $K(m)$ is not unique, not even up to isomorphism, but its $F$-polynomial is unique. In particular, the linear maps associated with arrows satisfy relations in the Jacobian ideal $J$.

The following several examples hold that $\chi^-_q(L(\prod_{i=1}^{T}Y_{i_t,k_t}))=\chi_q(L(\prod_{i=1}^{T}Y_{i_t,k_t}))$.

\begin{example}\label{example111}
In type $\mathbb{A}_3$, let $m=Y_{1,-15} Y_{3,-11} Y_{3,-9} Y_{2,-6}$. Then $K(m)$ is displayed in Figure~\ref{example11} (Here $K(m)$ is drawn opposite as Figure~\ref{example1}, because of the definition of paths). For each vertex $(i,r)\in V_m$, we find it convenient to always label the dimension of the vector space at the vertex $(i,r)$. The dimension associated with a vertex $(i,r)\in V_m$ is the multiplicity of cells with coordinate $(i,r)$ occurring in the multiset 
\[
(p^-_{1,-15}\ominus p^{+}_{1,-15})\cup (p^-_{3,-11}\ominus p^{+}_{3,-11})\cup (p^-_{3,-9}\ominus p^{+}_{3,-9}) \cup (p^-_{2,-6}\ominus p^{+}_{2,-6}).
\] 
The maps associated with arrows are $(\pm 1)$, whose sign is deduced from the defining relations of the Jacobian algebra $A$.

In the sense of Theorem \ref{F-polynomial formula}, finding all possible submodules of $K(m)$ is equivalent to finding all 4-tuple sets of non-overlapping paths in $\mathscr{P}_{(1,-15)}\times \mathscr{P}_{(3,-11)}\times\mathscr{P}_{(3,-9)}\times\mathscr{P}_{(2,-6)}$. 

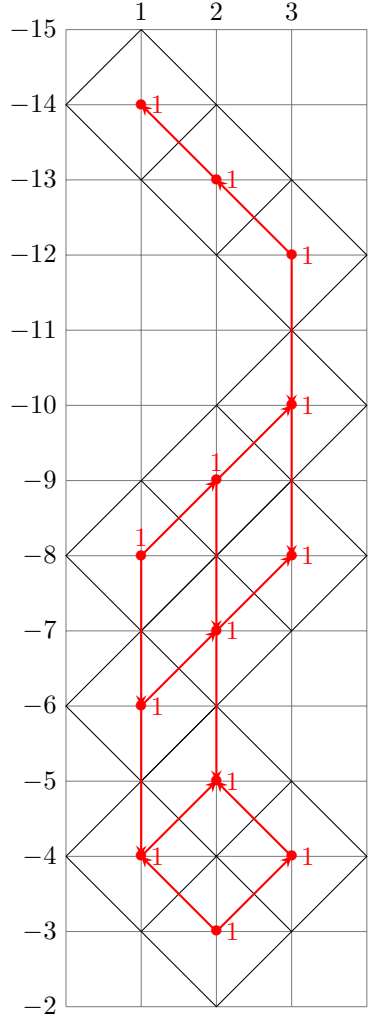
\begin{figure}
\resizebox{1.0\width}{1.0\height}{
\begin{minipage}[t]{.3\linewidth}
\begin{tikzpicture}
\draw[step=1cm,gray,thin] (0,1) grid (4,14);
\node[above] at (1,14) {$1$};
\node[above] at (2,14) {$2$};
\node[above] at (3,14) {$3$};
\node[left] at (0,1) {$-2$};
\node[left] at (0,2) {$-3$};
\node[left] at (0,3) {$-4$};
\node[left] at (0,4) {$-5$};
\node[left] at (0,5) {$-6$};
\node[left] at (0,6) {$-7$};
\node[left] at (0,7) {$-8$};
\node[left] at (0,8) {$-9$};
\node[left] at (0,9) {$-10$};
\node[left] at (0,10) {$-11$};
\node[left] at (0,11) {$-12$};
\node[left] at (0,12) {$-13$};
\node[left] at (0,13) {$-14$};
\node[left] at (0,14) {$-15$};
\node[red] at (1,13) {$\small \bullet$};
\node[red] at (2,12) {$\bullet$};
\node[red] at (3,11) {$\bullet$};
\node[red] at (3,9) {$\bullet$};
\node[red] at (2,8) {$\bullet$};
\node[red] at (1,7) {$\bullet$};
\node[red] at (3,7) {$\bullet$};
\node[red] at (2,6) {$\bullet$};
\node[red] at (1,5) {$\bullet$};
\node[red] at (2,2) {$\bullet$};
\node[red] at (2,4) {$\bullet$};
\node[red] at (1,3) {$\bullet$};
\node[red] at (3,3) {$\bullet$};
\node[red,right] at (1,13) {$1$};
\node[red,right] at (2,12) {$1$};
\node[red,right] at (3,11) {$1$};
\node[red,right] at (3,9) {$1$};
\node[red,above] at (2,8) {$1$};
\node[red,above] at (1,7) {$1$};
\node[red,right] at (3,7) {$1$};
\node[red,right] at (2,6) {$1$};
\node[red,right] at (1,5) {$1$};
\node[red,right] at (2,2) {$1$};
\node[red,right] at (2,4) {$1$};
\node[red,right] at (1,3) {$1$};
\node[red,right] at (3,3) {$1$};
\draw (1,14)--(2,13)--(3,12)--(4,11)--(3,10)--(2,11)--(1,12)--(0,13)--(1,14);
\draw (1,12)--(2,13) (2,11)--(3,12);
\draw (3,10)--(4,9)--(3,8)--(2,7)--(1,6)--(0,7)--(1,8)--(2,9)--(3,10);
\draw (2,7)--(1,8) (3,8)--(2,9);
\draw (3,8)--(4,7)--(3,6)--(2,5)--(1,4)--(0,5)--(1,6)--(2,7)--(3,8);
\draw (2,5)--(1,6) (3,6)--(2,7);
\draw (0,3)--(2,5)--(4,3)--(2,1)--(0,3);
\draw (1,2)--(3,4) (1,4)--(3,2);
\draw[->,thick,>=stealth,red] (2,2)--(1,3);
\draw[->,thick,>=stealth,red] (2,2) -- (3,3);
\draw[->,thick,>=stealth,red] (1,3) -- (2,4);
\draw[->,thick,>=stealth,red] (3,3) -- (2,4);
\draw[->,thick,>=stealth,red](2,12) -- (1,13);
\draw[->,thick,>=stealth,red] (3,11) -- (2,12);
\draw[->,thick,>=stealth,red] (3,11) -- (3,9);
\draw[->,thick,>=stealth,red] (2,8) -- (3,9);
\draw[->,thick,>=stealth,red] (1,7) -- (2,8);
\draw[->,thick,>=stealth,red] (1,5) -- (2,6);
\draw[->,thick,>=stealth,red] (2,6) -- (3,7);
\draw[->,thick,>=stealth,red] (1,7) -- (1,5);
\draw[->,thick,>=stealth,red] (1,5) -- (1,3);
\draw[->,thick,>=stealth,red] (2,8) -- (2,6);
\draw[->,thick,>=stealth,red] (2,6) -- (2,4);
\draw[->,thick,>=stealth,red] (3,9)--(3,7);
\end{tikzpicture}
\end{minipage}}
\caption{The module $K(Y_{1,-15} Y_{3,-11} Y_{3,-9} Y_{2,-6})$ in type $\mathbb{A}_3$. The index of $Y_{i,r}$ corresponds to the top vertex of the associated path rectangle $\mathscr{P}_{i,r}$. The dimension is $1$ at those vertices of $\Gamma^-$ that lie in the interior of these rectangles.}\label{example11}
\end{figure}
\end{example}

\begin{example}\label{example222}
In type $\mathbb{A}_3$, let $m=Y_{2,-10} Y_{2,-6}$. Then $K(m)$ is displayed in Figure \ref{example12} (Here $K(m)$ is drawn opposite as Figure \ref{example2}, because of the definition of paths). For each vertex $(i,r)\in V_m$, we label the dimension of the vector space at the vertex $(i,r)$. The dimension associated with a vertex $(i,r)\in V_m$ is the multiplicity of cells with coordinate $(i,r)$ occurring in the multiset 
\[
(p^-_{2,-10}\ominus p^{+}_{2,-10})\cup (p^-_{2,-6}\ominus p^{+}_{2,-6}). 
\]
The maps associated with arrows are $(\pm 1)$, whose sign is deduced from the defining relations of the Jacobian algebra $A$.

In the sense of Theorem \ref{F-polynomial formula}, finding all possible submodules of $K(m)$ is equivalent to finding all pairs of non-overlapping paths in $\mathscr{P}_{(2,-10)}\times\mathscr{P}_{(2,-6)}$.

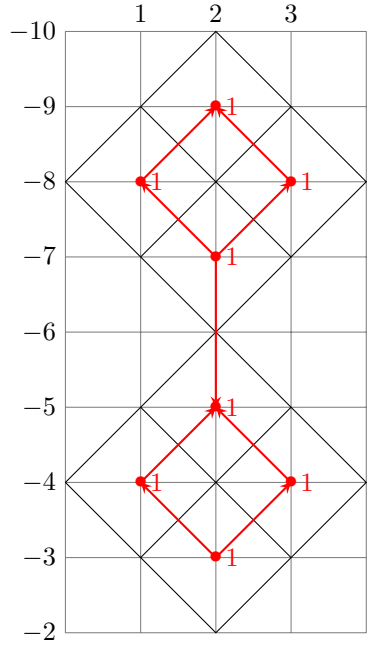
\begin{figure}
\resizebox{1.0\width}{1.0\height}{
\begin{minipage}[t]{.3\linewidth}
\begin{tikzpicture}
\draw[step=1cm,gray,thin] (0,-1) grid (4,7);
\node[above] at (1,7) {$1$};
\node[above] at (2,7) {$2$};
\node[above] at (3,7) {$3$};
\node[left] at (0,-1) {$-2$};
\node[left] at (0,0) {$-3$};
\node[left] at (0,1) {$-4$};
\node[left] at (0,2) {$-5$};
\node[left] at (0,3) {$-6$};
\node[left] at (0,4) {$-7$};
\node[left] at (0,5) {$-8$};
\node[left] at (0,6) {$-9$};
\node[left] at (0,7) {$-10$};
\node[red] at (2,0) {$\bullet$};
\node[red] at (2,2) {$\bullet$};
\node[red] at (2,4) {$\bullet$};
\node[red] at (2,6) {$\bullet$};
\node[red] at (1,5) {$\bullet$};
\node[red] at (1,1) {$\bullet$};
\node[red] at (3,5) {$\bullet$};
\node[red] at (3,1) {$\bullet$};
\node[red,right] at (2,0) {$1$};
\node[red,right] at (2,2) {$1$};
\node[red,right] at (2,4) {$1$};
\node[red,right] at (2,6) {$1$};
\node[red,right] at (1,5) {$1$};
\node[red,right] at (1,1) {$1$};
\node[red,right] at (3,5) {$1$};
\node[red,right] at (3,1) {$1$};
\draw (0,5)--(2,7)--(4,5)--(2,3)--(0,5);
\draw (0,1)--(2,3)--(4,1)--(2,-1)--(0,1);
\draw (1,6)--(3,4) (1,4)--(3,6) (1,2)--(3,0) (1,0)--(3,2);
\draw[->,thick,>=stealth,red] (2,4) -- (1,5);
\draw[->,thick,>=stealth,red] (2,0) -- (1,1);
\draw[->,thick,>=stealth,red] (2,4) -- (3,5);
\draw[->,thick,>=stealth,red] (2,0) -- (3,1);
\draw[->,thick,>=stealth,red] (1,5) -- (2,6);
\draw[->,thick,>=stealth,red] (1,1) -- (2,2);
\draw[->,thick,>=stealth,red] (3,5) -- (2,6);
\draw[->,thick,>=stealth,red] (3,1) -- (2,2);
\draw[->,thick,>=stealth,red] (2,4) -- (2,2);
\end{tikzpicture}
\end{minipage}}
\caption{The module $K(Y_{2,-10}Y_{2,-6})$ in type $\mathbb{A}_3$. The index of $Y_{i,r}$ corresponds to the top vertex of the associated path rectangle $\mathscr{P}_{i,r}$. The dimension is $1$ at those vertices of $\Gamma^-$ that lie in the interior of these rectangles.}\label{example12}
\end{figure}
\end{example}

\begin{example}\label{example333}
In type $\mathbb{B}_2$, let $m=Y_{2,-12} Y_{2,-6}$. Then $K(m)$ is displayed in Figure \ref{example13} (Here $K(m)$ is drawn opposite as Figure \ref{example3}, because of the definition of paths). For each vertex $(i,r)\in V_m$, we label the dimension of the vector space at the vertex $(i,r)$. The dimension associated with a vertex $(i,r)\in V_m$ is the multiplicity of cells with coordinate $(i,r)$ occurring in the multiset 
\[
(p^-_{2,-12}\ominus p^{+}_{2,-12})\cup (p^-_{2,-6}\ominus p^{+}_{2,-6}).
\]
The maps associated with arrows are $(\pm 1)$, whose sign is deduced from the defining relations of the Jacobian algebra $A$.

In the sense of Theorem \ref{F-polynomial formula}, finding all possible submodules of $K(m)$ is equivalent to finding all pairs of non-overlapping paths in $\mathscr{P}_{(2,-12)}\times\mathscr{P}_{(2,-6)}$.

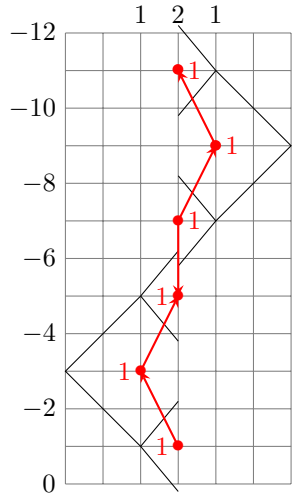
\begin{figure}
\resizebox{1.0\width}{1.0\height}{
\begin{minipage}[t]{.3\linewidth}
\begin{tikzpicture}
\draw[step=0.5cm,gray,thin] (0,-1) grid (3,5);
\node[above] at (1,5) {$1$};
\node[above] at (1.5,5) {$2$};
\node[above] at (2,5){$1$};
\node[left] at (0,-1) {$0$};
\node[left] at (0,0) {$-2$};
\node[left] at (0,1) {$-4$};
\node[left] at (0,2) {$-6$};
\node[left] at (0,3) {$-8$};
\node[left] at (0,4) {$-10$};
\node[left] at (0,5) {$-12$};
\draw (1.5,5.1)--(2,4.5)--(3,3.5)--(2,2.5)--(1.5,1.9);
\draw (2,4.5)--(1.5,3.9) (1.5,3.1)--(2,2.5);
\draw (1.5,2.1)--(1,1.5)--(0,0.5)--(1,-0.5)--(1.5,-1.1);
\draw (1,1.5)--(1.5,0.9)  (1,-0.5)--(1.5,0.1);
\node[red] at (1.5,4.5) {$\bullet$};
\node[red] at (2,3.5) {$\bullet$};
\node[red] at (1.5,2.5) {$\bullet$};
\node[red] at (1.5,4.5) {$\bullet$};
\node[red] at (2,3.5) {$\bullet$};
\node[red] at (1.5,1.5) {$\bullet$};
\node[red] at (1,0.5) {$\bullet$};
\node[red] at (1.5,-0.5) {$\bullet$};
\node[red,right] at (1.5,4.5) {$1$};
\node[red,right] at (2,3.5) {$1$};
\node[red,right] at (1.5,2.5) {$1$};
\node[red,left] at (1.5,1.5) {$1$};
\node[red,left] at (1,0.5) {$1$};
\node[red,left] at (1.5,-0.5) {$1$};
\draw[->,thick,>=stealth,red] (1.5,2.5)--(2,3.5);
\draw[->,thick,>=stealth,red] (2,3.5)--(1.5,4.5);
\draw[->,thick,>=stealth,red] (1.5,2.5)--(1.5,1.5);
\draw[->,thick,>=stealth,red] (1.5,-0.5)--(1,0.5);
\draw[->,thick,>=stealth,red] (1,0.5)--(1.5,1.5);
\end{tikzpicture}
\end{minipage}}
\caption{The module $K(Y_{2,-12}Y_{2,-6})$ in type $\mathbb{B}_2$. The index of $Y_{i,r}$ corresponds to the top vertex of the associated path triangle $\mathscr{P}_{i,r}$. The dimension is $1$ at those vertices of $\Gamma^-$ that lie in the interior of these triangles.}\label{example13}
\end{figure}
\end{example}

The following is an example where the dimensions of $K(m)$ are larger than $1$.
\begin{example}
In type $\mathbb{A}_3$, let $m=Y_{2,-8}Y_{2,-6}$. Then $L(m)$ is a Kirillov-Reshetikhin module and $K(m)$ is displayed in Figure \ref{example14}. For each vertex $(i,r)\in V_m$, we label the dimension of the vector space at the vertex $(i,r)$. The dimension associated with a vertex $(i,r)\in V_m$ is the multiplicity of cells with coordinate $(i,r)$ occurring in the multiset $(p^-_{2,-8}\ominus p^{+}_{2,-8})\cup (p^-_{2,-6}\ominus p^{+}_{2,-6})$. In Figure \ref{example14}, almost all vertices carry a vector space of dimension 1, except the vertex $(2,-5)$ which carries a vector space of dimension 2.

Starting from the initial seed $({\bf z}, G^-)$, the following sequence of mutations produces (in the last step) the cluster variable corresponding to $L(m)$.
\begin{align*}
(2,-2), (2,-4), (2,-6), (1,-3), (1,-5), (3,-3), (3,-5), (2,-2), (2,-4).    
\end{align*} 

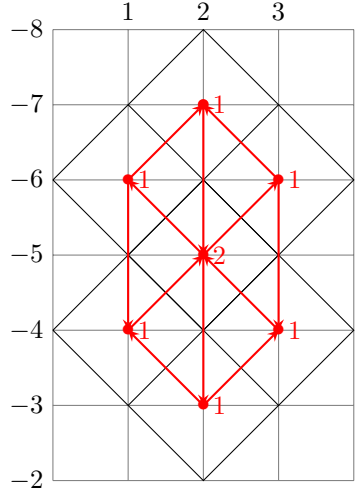
\begin{figure}
\resizebox{1.0\width}{1.0\height}{
\begin{minipage}[t]{.3\linewidth}
\begin{tikzpicture}
\draw[step=1cm,gray,thin] (0,-1) grid (4,5);
\node[above] at (1,5) {$1$};
\node[above] at (2,5) {$2$};
\node[above] at (3,5) {$3$};
\node[left] at (0,-1) {$-2$};
\node[left] at (0,0) {$-3$};
\node[left] at (0,1) {$-4$};
\node[left] at (0,2) {$-5$};
\node[left] at (0,3) {$-6$};
\node[left] at (0,4) {$-7$};
\node[left] at (0,5) {$-8$};
\node[red] at (2,0) {$\bullet$};
\node[red] at (2,2) {$\bullet$};
\node[red] at (2,4) {$\bullet$};
\node[red] at (2,4) {$\bullet$};
\node[red] at (1,3) {$\bullet$};
\node[red] at (1,1) {$\bullet$};
\node[red] at (3,3) {$\bullet$};
\node[red] at (3,1) {$\bullet$};
\node[red,right] at (2,0) {$1$};
\node[red,right] at (2,2) {$2$};
\node[red,right] at (2,4) {$1$};
\node[red,right] at (1,3) {$1$};
\node[red,right] at (1,1) {$1$};
\node[red,right] at (3,3) {$1$};
\node[red,right] at (3,1) {$1$};
\draw (0,3)--(2,5)--(4,3)--(2,1)--(0,3);
\draw (0,1)--(2,3)--(4,1)--(2,-1)--(0,1);
\draw (1,4)--(3,2) (1,2)--(3,4) (1,2)--(3,0) (1,0)--(3,2);
\draw[->,thick,>=stealth,red] (2,2) -- (1,3);
\draw[->,thick,>=stealth,red] (2,0) -- (1,1);
\draw[->,thick,>=stealth,red] (2,2) -- (3,3);
\draw[->,thick,>=stealth,red] (2,0) -- (3,1);
\draw[->,thick,>=stealth,red] (1,3) -- (2,4);
\draw[->,thick,>=stealth,red] (1,1) -- (2,2);
\draw[->,thick,>=stealth,red] (3,3) -- (2,4);
\draw[->,thick,>=stealth,red] (3,1) -- (2,2);
\draw[->,thick,>=stealth,red] (1,3) -- (1,1);
\draw[->,thick,>=stealth,red] (3,3) -- (3,1);
\draw[->,thick,>=stealth,red] (2,4) -- (2,2);
\draw[->,thick,>=stealth,red] (2,2) -- (2,0);
\end{tikzpicture}
\end{minipage}}
\caption{The module $K(Y_{2,-8}Y_{2,-6})$ in type $\mathbb{A}_3$. The index of $Y_{i,r}$ corresponds to the top vertex of the associated path rectangle $\mathscr{P}_{i,r}$. The dimension is $1$ at those vertices of $\Gamma^-$ that lie in the interior of these rectangles, except that the vertex $(2,-5)$ has a vector space of dimension $2$.}\label{example14}
\end{figure}
\end{example}

\begin{remark}
The dimension of $K(m)$ at a vertex $(i,r)$ can be arbitrary large in the sense that given any integer $\alpha$ there is a snake module $L(m)$ and a vertex $(i,r)$ such that the generic kernel $K(m)$ is of dimension at least $\alpha$ at $(i,r)$. Therefore Corollary \ref{corollary of Grassmannian} is non-trival.
\end{remark}

The following is an example that $\chi^-_q(L(\prod_{i=1}^{T}Y_{i_t,k_t}))\neq \chi_q(L(\prod_{i=1}^{T}Y_{i_t,k_t}))$.

\begin{example}
In type $\mathbb{A}_3$, let $m=Y_{2,-4}Y_{2,-2}$. Then $L(m)$ is a Kirillov-Reshetikhin module and $K(m)$ is displayed in Figure \ref{example15}. By definition, we have 
\[
V_m=({p'}^-_{2,-4}\ominus p^{+}_{2,-4})\cup ({p'}^-_{2,-2}\ominus p^{+}_{2,-2})=\{(2,-3),(2,-1)\},
\]
where ${p'}^-_{2,-2}=\{(0,0), (1,-1), (2,0),(3,-1),(4,0)\}$ and 
\[
{p'}^-_{2,-4}=\{(0,-2), (1,-3), (2,-2),(3,-3),(4,-2)\}.
\] 
Note that ${p'}^-_{2,-2}$ is a path in the set $\mathscr{P}'_{2,-2}$, so it cannot go through points $(i,r)$ with $r>0$.

For each vertex $(i,r)\in V_m$, we label the dimension of the vector space at the vertex $(i,r)$. The dimension associated with a vertex $(i,r)\in V_m$ is the multiplicity of cells with coordinate $(i,r)$ occurring in the multiset $({p'}^-_{2,-4}\ominus p^{+}_{2,-4})\cup ({p'}^-_{2,-2}\ominus p^{+}_{2,-2})$. In Figure \ref{example15}, all vertices carry a vector space of dimension 1.

Starting from the initial seed $({\bf z}, G^-)$, the sequence $((2,0), (2,-2))$ of mutations produces (in the last step) the cluster variable corresponding to $L(m)$.

\begin{figure}
\resizebox{1.0\width}{1.0\height}{
\begin{minipage}[t]{.3\linewidth}
\begin{tikzpicture}
\draw[step=1cm,gray,thin] (0,-1) grid (4,5);
\node[above] at (1,5) {$1$};
\node[above] at (2,5) {$2$};
\node[above] at (3,5) {$3$};
\node[left] at (0,-1) {$2$};
\node[left] at (0,0) {$1$};
\node[left] at (0,1) {$0$};
\node[left] at (0,2) {$-1$};
\node[left] at (0,3) {$-2$};
\node[left] at (0,4) {$-3$};
\node[left] at (0,5) {$-4$};
\node[red] at (2,2) {$\bullet$};
\node[red] at (2,4) {$\bullet$};
\node[red,right] at (2,2) {$1$};
\node[red,right] at (2,4) {$1$};
\draw (0,3)--(2,5)--(4,3)--(2,1)--(0,3);
\draw (0,1)--(2,3)--(4,1) (1,2)--(2,1)--(3,2);
\draw (1,4)--(3,2) (1,2)--(3,4);
\draw[thick,->,>=stealth,red] (2,4) -- (2,2);
\draw[dashed] (0,1)--(2,-1)--(4,1) (1,0)--(2,1)--(3,0);
\end{tikzpicture}
\end{minipage}}
\caption{The module $K(Y_{2,-4}Y_{2,-2})$ in type $\mathbb{A}_3$. The index of $Y_{i,r}$ corresponds to the top vertex of the associated path rectangle $\mathscr{P}_{i,r}$. The dimension is $1$ at vertices $(2,-3)$ and $(2,-1)$.}\label{example15}
\end{figure}
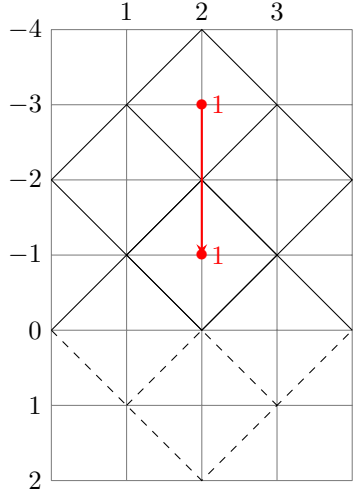
\end{example}

\section{Denominator vector}\label{denominator vector}
In this section, we show that every snake module corresponds to a cluster monomial with square free denominator in the cluster algebra $\mathscr{A}$ and that snake modules are real modules. 

\begin{theorem}\label{cluster monomials with square free denominator}
Let $L(m)$ be an arbitrary snake module. Then the truncated $q$-character $\chi_q^-(L(m))$ is a cluster monomial in $\mathscr{A}$, and its denominator is square free as a monomial in the initial cluster variables $z_{i,r}, (i,r)\in G^-_0$.
\end{theorem}

\begin{proof}
By Theorem \ref{factorization of snake modules}, we can write $L(m)$ as a tensor product $L(m)\cong L(m_1) \otimes \cdots \otimes L(m_n)$ of prime snake modules. Let $K(m_i)$ be the generic kernel associated to $L(m_i)$ and let $Q(m_i)$ be the full subquiver of $\Gamma^-$ whose vertices are in the support of $K(m_i)$. Thus $K(m)=\bigoplus_{i=1}^n K(m_i)$ is the generic kernel associated to $L(m)$. To show that $L(m)$ corresponds to a cluster monomial we need to prove that $K(m)$ is a rigid object in the cluster category \cite{Am09,FK10}.

Let $\mathscr{P}_{(i_t(m_i),k_t(m_i))_{1\leq t \leq T_i}}$ be the set of paths associated to $L(m_i)$. By Theorem \ref{path description of q-characters} and Remark~\ref{overlapping property of snake modules}, we know that for any $1\leq i\neq j\leq n$, the sets $\mathscr{P}_{(i_t(m_i),k_t(m_i))_{1\leq t \leq T_i}}$ and $\mathscr{P}_{(i_t(m_j),k_t(m_j))_{1\leq t \leq T_j}}$ are non-overlapping. By our construction in Section \ref{generic kernel}, this implies that the quivers $Q(m_i)$ and $Q(m_j)$ are disjoint and there are no arrows in $\Gamma^-$ which connect $Q(m_i)$ and $Q(m_j)$.

Definition \ref{definition of generic kernel} and Proposition \ref{g-vectors} imply that for the prime snake module $L(m_j)$, we have 
\[
I(m_j)^+=\bigoplus_{g_{\ell,s}(m_j)=1} I_{\ell,s-d_\ell}, \quad  I(m_j)^-=\bigoplus_{g_{\ell,s}(m_j)=-1} I_{\ell,s-d_\ell}.
\]
By Section \ref{generic kernel}, the support of $K(m_i)$ is contained in $\mathscr{P}_{(i_t(m_i),k_t(m_i))_{1\leq t \leq T_i}}$. Thus the socle points $(\ell,s-d_\ell)$ in $I(m_j)^+$ cannot be in the support of $K(m_i)$. Otherwise, the set $\mathscr{P}_{\ell,s}$ for $L(m_j)$ and $\mathscr{P}_{(i_t(m_i),k_t(m_i))_{1\leq t \leq T_i}}$ would be overlapping (there is at least a common vertex $(\ell,s)$). This is a contradiction to the fact that $L(m_i)\otimes L(m_j)$ is not prime, see Remark \ref{overlapping property of snake modules}. 

Therefore 
\begin{align}\label{Hom=0}
\begin{split}
\text{Hom}_A(K(m_i),I(m_j)^+) & =\text{Hom}_A(K(m_i),\bigoplus_{g_{\ell,s}(m_j)=1} I_{\ell,s-d_\ell}) \\
& \cong \bigoplus_{g_{\ell,s}(m_j)=1}(K(m_i))_{\ell,s-d_\ell}=0.
\end{split}
\end{align}
Similarly, $\text{Hom}_A(K(m_j),I(m_i)^+)=0$.     

Consider the injective resolution 
\begin{align*}
\xymatrix{
0 \ar[r] & K(m_j) \ar[r]^{i_0} & I^-(m_j) \ar[r]^{i_1} & I(m_j)^+ \ar[r]^{i_2} & \cdots.} 
\end{align*}
Then $\text{Ext}^1_A(K(m_i),K(m_j))$ is a quotient of $\{f\in \Hom(K(m_i),I(m_j)^+) \mid i_2 f=0\}$ which is zero by~(\ref{Hom=0}). Thus 
\[
\text{Ext}^1_A(K(m_i),K(m_j))=0.
\]
Similarly, $\text{Ext}^1_A(K(m_j),K(m_i))=0$.

By Corollary \ref{Km rigid}, we have that $\text{Ext}^1_A(K(m_i),K(m_i))=0$ for any $1\leq i\leq n$. In \cite{Am09}, $K(m_i)$ and $K(m_j)$ are compatible if and only if 
\[
\text{Ext}^1_{\mathcal{C}}(K(m_i),K(m_j))=0,
\]
where $\mathcal{C}$ is the (generalized) cluster category of the Jacobian algebra $A$.

Applying 
\[
\text{Ext}^1_{\mathcal{C}}(K(m_i),K(m_j)) \cong \text{Ext}^1_A(K(m_i),K(m_j)) \bigoplus \text{Ext}^1_A(K(m_j),K(m_i)),
\]
we see that $K(m_i)$ and $K(m_j)$ are compatible for all $i,j=1,\ldots, n$, and hence snake modules are cluster monomials.

Next we prove the statement about square free denominators using the Mukhin-Young's formulas in Theorem \ref{path description of q-characters}. By Theorem \ref{F-polynomial formula} and its proof, we have
\begin{align*}
\chi^-_q(L(m))=m F_{K(m)} = \sum_{(p_1,\ldots,p_T)\in \overline{\mathscr{P}}_{(i_t,k_t)_{1\leq t\leq T}}} m \prod_{t=1}^T h(p_t).
\end{align*}
For any $T$-tuple $(p_1,\ldots,p_T)$ of non-overlapping paths, either $m \prod_{t=1}^T h(p_t) =0$ or  by Theorem~\ref{path description of q-characters}, 
\begin{align}\label{denominator general form}
\begin{split}
\ m \prod_{t=1}^T h(p_t) = \prod_{t=1}^T\mathfrak{m}(p_t) & =\prod_{t=1}^T \left(\prod_{(j,\ell)\in C^{+}_{p_t}}{\hskip -0.5em}Y_{j,\ell}{\hskip -0.5em}\prod_{(j,\ell)\in C^{-}_{p_t}}{\hskip -0.5em}Y^{-1}_{j,\ell}\right) \\
& = \prod_{t=1}^T \left(\prod_{(j,\ell)\in C^{+}_{p_t}} \frac{z_{j,\ell}}{z_{j,\ell+b_{jj}}}  \prod_{(j,\ell)\in C^{-}_{p_t}}{\hskip -0.5em} \frac{z_{j,\ell+b_{jj}}}{z_{j,\ell}} \right),
\end{split}
\end{align}
where the last equation is obtained by performing the change of variables (\ref{variable substitution z-Y}). It is obvious that 
\[
\prod_{(j,\ell)\in C^{+}_{p_t}} \frac{z_{j,\ell}}{z_{j,\ell+b_{jj}}}  \prod_{(j,\ell)\in C^{-}_{p_t}}{\hskip -0.5em} \frac{z_{j,\ell+b_{jj}}}{z_{j,\ell}}
\]
is a fraction with square free denominator in the initial cluster variables $z_{i,r}, (i,r)\in G^-_0$.

For any $1\leq t_1 \neq t_2 \leq T$, the expression 
\begin{align}\label{square free denominator}
\left(\prod_{(j,\ell)\in C^{+}_{p_{t_1}}} \frac{z_{j,\ell}}{z_{j,\ell+b_{jj}}}  \prod_{(j,\ell)\in C^{-}_{p_{t_1}}}{\hskip -0.5em} \frac{z_{j,\ell+b_{jj}}}{z_{j,\ell}} \right) \left(\prod_{(j,\ell)\in C^{+}_{p_{t_2}}} \frac{z_{j,\ell}}{z_{j,\ell+b_{jj}}}  \prod_{(j,\ell)\in C^{-}_{p_{t_2}}}{\hskip -0.5em} \frac{z_{j,\ell+b_{jj}}}{z_{j,\ell}} \right)
\end{align}
is still a fraction with square free denominator. Otherwise either $z_{j,\ell+b_{jj}}$ for some $(j,\ell)\in C^{+}_{p_{t_1}}$ or $z_{j,\ell}$ for some $(j,\ell)\in C^{-}_{p_{t_1}}$ in the first term also appear in the denominator of the second term. If $(j,\ell)\in C^{+}_{p_{t_2}}$, then $p_{t_1}$ and $p_{t_2}$ overlap at least at the vertex $(j,\ell)$. If $(j,\ell)\in C^{-}_{p_{t_2}}$, then $p_{t_1}$ and $p_{t_2}$ overlap at least at a vertex $(i,r)$, where $\iota^{-1}(i,r)=(j\pm 1,\ell+1)$ or $(j\pm 1,\ell+2)$. This is a contradiction. Similarly we deal with $(j,\ell)\in C^{-}_{p_{t_1}}$.

Therefore (by induction) $\chi^-_q(L(m))$ is a Laurent polynomial with square free denominator in the initial cluster variables. 
\end{proof}

Recall that prime snake modules are prime, real modules, see Theorem \ref{factorization of snake modules}. As a slight generalization, we have the following theorem.
\begin{theorem}\label{SM-real modules}
Snake modules are real simple modules.
\end{theorem}
\begin{proof}
We assume that $L(m)$ is a snake module. Then $L(m)=L(m_1)\otimes \cdots \otimes L(m_n)$ with $L(m_i)$ prime by Theorem \ref{factorization of snake modules}. We only need to show that snake modules are real. 

Using the fact that $\chi_q$ is a ring homomorphism, we have  
\begin{align*}
\chi_q(L(m)\otimes L(m)) & = \chi_q(L(m)) \chi_q(L(m)) \\
& = \chi_q(L(m_1))\cdots \chi_q(L(m_n)) \chi_q(L(m_1))\cdots \chi_q(L(m_n))\\
& = \chi_q(L(m_1)\otimes L(m_1))\cdots \chi_q(L(m_n)\otimes L(m_n)).
\end{align*}

By Theorem 3.4 of \cite{DLL19}, we have the fact that for every $1\leq \ell\leq n$, $\chi_q(L(m_\ell)\otimes L(m_\ell))$ has only one dominant monomial $m_\ell^2$.

Using Theorem \ref{path description of q-characters} and Remark \ref{overlapping property of snake modules}, for any $1\leq i\neq j\leq n$, we see that 
\[
\mathscr{P}_{(i_t(m_i),k_t(m_i))_{1\leq t \leq T_i}} \text{ and } \mathscr{P}_{(i_t(m_j),k_t(m_j))_{1\leq t \leq T_j}}
\]
are non-overlapping. So monomials with negative exponents occurring in $\chi_q(L(m_i)\otimes L(m_i))$ cannot be canceled by any monomial occurring in $\chi_q(L(m_j)\otimes L(m_j))$. Thus 
\[
\chi_q(L(m))\chi_q(L(m))=\chi_q(L(m)\otimes L(m))
\]
has only one dominant monomial $m^2$. This shows that $L(m)\otimes L(m)$ is simple, and thus $L(m)$ is real.
\end{proof}

\begin{remark}
\begin{itemize}
\item[(1)] From Proposition \ref{g-vectors}, it follows that for different snake modules, the corresponding cluster monomials have different ${\bf g}$-vectors with respect to a given initial seed.  
\item[(2)] Combining Theorem \ref{cluster monomials with square free denominator} and Theorem \ref{SM-real modules}, we give a partial answer of Conjecture \ref{conj 5.2}.
\end{itemize}
\end{remark}

Recall that $\mathscr{A}$ is the cluster algebra introduced by Hernandez and Lerclec in \cite{HL16}, also see Section~\ref{section2.1}. In \cite{HL16}, Hernandez and Leclerc applied the method of cluster mutations to give an algorithm for computing the $q$-characters of Kirillov-Reshetikhin modules by successive approximations. 

An explicit formula of the expansion for snake modules even Kirillov-Reshetikhin modules in terms of the initial cluster variables is usually very complicated. In the following theorem, we give explicitly the denominator of the cluster monomial associated to a snake module $L(m)$.

\begin{theorem}\label{denominator formula}
Suppose that $L(m)$ is a snake module. Then the denominator of the cluster monomial associated to $L(m)$ is 
\begin{align}\label{expression of denominator formula}
d(\chi^-_q(L(m)))=\prod_{(i,r-d_i) \in \text{Supp}(K(m))} z_{i,r}.
\end{align}
\end{theorem}

\begin{proof}
We first show that for any $(i,r-d_i) \in \text{Supp}(K(m))$, the variable $z_{i,r}$ appears in the denominator of the cluster monomial associated to $L(m)$.  

We assume without loss of generality that $m=\prod_{i=1}^{T}Y_{i_t,k_t}$ for some snake $(i_t,k_t)_{1\leq t\leq T}$. The cluster monomial of $L(m)$ is given by the truncated $q$-character $\chi^-_q(L(m))$ after the change of variables (\ref{variable substitution z-Y}). For any $(i,r)$ such that $(i,r-d_i) \in \text{Supp}(K(m))$, there exists a path $p_t\in \mathscr{P}'_{i_t,k_t}$ such that $(i,r)$ is the unique lower corner of $p_t$. We choose $p_t$ such that $t$ is maximal. Define $p_\ell \in \mathscr{P}'_{i_\ell,k_\ell}$ by 
\begin{align*}
p_\ell = \begin{cases} 
p^{+}_{i_{\ell},k_{\ell}} & 1\leq \ell <t, \\
{p'}^{-}_{i_{\ell},k_{\ell}} & t<\ell\leq T.
\end{cases}
\end{align*}
Then the $T$-tuple $(p_1,\ldots,p_T)$ is a set of non-overlapping paths, because $p_t$ is a path between $p^{+}_{i_t,k_t}$ and ${p'}^{-}_{i_t,k_t}$.

Using equation (\ref{proof of F-polynomial formula-6}), we have 
\begin{align*} 
\chi^-_q(L(m))= \sum_{\substack{(p_{1},\ldots,p_{T}) \in \overline{\mathscr{P}}_{(i_{\ell},k_{\ell})_{1 \leq \ell \leq T}}\\ C^+_{p_\ell}, C^-_{p_\ell}\subset G^-_0}} \prod_{\ell=1}^{T}\mathfrak{m}(p_{\ell}).
\end{align*}

After performing the change of variables (\ref{variable substitution z-Y}), the variable $z_{i,r}$ appears in the denominator of $\prod_{\ell=1}^{T}\mathfrak{m}(p_{\ell})$. Thus the product (\ref{expression of denominator formula}) appears in the denominator of the cluster monomial $\chi^-_q(L(m))$. 

On the other hand, for any $T$-tuple $(p_1,\ldots,p_T)$ of non-overlapping paths, we have the following: If $z_{i,r}$, $(i,r)\in G^-_0$, appears in the denominator of $\prod_{\ell=1}^{T}\mathfrak{m}(p_{\ell})$,  then $(i,r)\in C^-_{p_\ell}$ or $(i,r-b_{ii})\in C^+_{p_\ell}$ for some $\ell$ by equation (\ref{denominator general form}). For the lower corner $(i,r)\in C^-_{p_\ell}$, we have $(i,r-d_i)\in \text{Supp}(K(m))$, by Remark (\ref{compute dimensional vector})~(1). For the upper corner $(i,r-b_{ii})\in C^+_{p_\ell}$, we need to modify our path $p_\ell$ by replacing the point $(i,r-b_{ii})$ by the point $(i,r)$. Note that $(i,r)\in  G^-_0$, so the modified path is still in $\mathscr{P}'_{(i_\ell,k_\ell)}$. Now Remark (\ref{compute dimensional vector})~(1) implies $(i,r-d_i)\in \text{Supp}(K(m))$.
\end{proof}

We illustrate this result in our five running examples. Note that for each vertex $(i,r-d_i)$ in the support of $K(m)$, as shown in Figures \ref{example11}--\ref{example15}, we have a contribution $z_{i,r}$ in the denominator. Here $d_i=1$ in type $\mathbb{A}$ and $d_1=2$ and $d_2=1$ in type $\mathbb{B}_2$. 

\begin{example}
In type $\mathbb{A}_3$, let $m=Y_{1,-15} Y_{3,-11} Y_{3,-9} Y_{2,-6}$. Then by Theorem \ref{denominator formula}, 
\[
d(\chi^-_q(L(m)))= z_{2,-2}z_{1,-3}z_{3,-3}z_{2,-4} z_{1,-5}z_{2,-6}z_{3,-7}z_{1,-7}z_{2,-8}z_{3,-9}z_{3,-11}z_{2,-12}z_{1,-13}.
\]
\end{example}

\begin{example}
In type $\mathbb{A}_3$, let $m=Y_{2,-10} Y_{2,-6}$. Then by Theorem \ref{denominator formula},
\[
d(\chi^-_q(L(m)))= z_{2,-2}z_{1,-3}z_{3,-3}z_{2,-4} z_{2,-6}z_{1,-7}z_{3,-7}z_{2,-8}.
\]
\end{example}

\begin{example}
In type $\mathbb{B}_2$, let $m=Y_{2,-12} Y_{2,-6}$. Then by Theorem \ref{denominator formula},  
\[
d(\chi^-_q(L(m)))=z _{2,0}z _{1,-1}z _{2,-4}z _{2,-6}z _{1,-7}z_{2,-10}.
\]
\end{example}

\begin{example}
In type $\mathbb{A}_3$, let $m=Y_{2,-8}Y_{2,-6}$. Then by Theorem \ref{denominator formula},
\[
d(\chi^-_q(L(m)))=z_{2,-2}z_{1,-3}z_{3,-3}z_{2,-4}z_{1,-5}z_{3,-5} z_{2,-6}.
\]
\end{example}

\begin{example}
In type $\mathbb{A}_3$, let $m=Y_{2,-4}Y_{2,-2}$. Then by Theorem \ref{denominator formula},
\[
d(\chi^-_q(L(m)))=z_{2,0}z_{2,-2}.
\]
\end{example}

It is natural to ask whether all cluster variables with square free denominator are always prime snake modules. The answer is No. The following example shows that there exists a module that is not a snake module and such that its truncated $q$-character corresponds to a cluster variable with square free denominator. 

\begin{example}\label{example4.8}
In type $\mathbb{A}_3$, let $m=Y_{1,-3}Y_{2,0}Y_{3,-3}$. This is not a snake module. Because the second coordinates in the indices do not form an increasing sequence. By Example 12.2 of \cite{HL10}, we have 
\[
[L(Y_{1,-3}Y_{2,0}Y_{3,-3})] [L(Y_{2,0})] = [L(Y_{1,-3}Y_{2,0})][L(Y_{3,-3}Y_{2,0})] + [L(Y_{2,-2}Y_{2,0})].
\]
Thus by Theorem 5.1 of \cite{HL16}
\[
\chi^-_q(L(m))=\frac{\chi^-_q(L(Y_{1,-3}Y_{2,0}))\chi^-_q(L(Y_{3,-3}Y_{2,0}))+\chi^-_q(L(Y_{2,-2}Y_{2,0}))}{\chi^-_q(L(Y_{2,0}))}.
\]
On the right hand side of the equation, every truncated $q$-character is known by the Frenkel-Mukhin algorithm, so  
\[
\chi^-_q(L(m)) = m(1+A^{-1}_{1,-2}+A^{-1}_{3,-2}+A^{-1}_{1,-2}A^{-1}_{3,-2}+A^{-1}_{1,-2}A^{-1}_{3,-2}A^{-1}_{2,-1}).
\]

The corresponding cluster variable $x_m$ is 
\[
x_m = \frac{(z_{1,-1}z_{3,-1}+z_{2,0}z_{1,-3})z_{2,-2}+z^2_{2,-2}+(z^2_{2,0}z_{1,-3}+z_{2,0}z_{2,-2})z_{3,-3}}{z_{1,-1}z_{2,0}z_{3,-1}},
\]
which has a square free denominator.
\end{example}

Finally, we point out that there exists a module beyond snake modules in $\mathscr{C}^-$ for which Hernandez and Leclerc's conjectural geometric formula holds.

\begin{example}\label{example compare HL with MY}
In type $\mathbb{A}_3$, let $m=Y_{1,-7}Y_{2,-4}Y_{3,-7}$. By \cite{HL10}, we know that $L(m)$ corresponds to a cluster variable in $\mathscr{A}$ (up to scalar), equivalently, its truncated $q$-character is a cluster variable in $\mathscr{A}$. Let
\begin{align*}
I(m)^+ = I_{1,-8} \oplus I_{2,-5} \oplus I_{3,-8}, \quad I(m)^- = I_{1,-6} \oplus I_{2,-3} \oplus I_{3,-6}.
\end{align*}
By Example 12.2 of \cite{HL10}, we have 
\[
[L(Y_{1,-7}Y_{2,-4}Y_{3,-7})] [L(Y_{2,-4})] = [L(Y_{1,-7}Y_{2,-4})][L(Y_{3,-7}Y_{2,-4})] + [L(Y_{2,-6}Y_{2,-4})].
\]
With the exception of $L(Y_{1,-7}Y_{2,-4}Y_{3,-7})$, those modules are minimal affinizations \cite{C95}, and we can compute their $q$-characters by the Frenkel-Mukhin algorithm. 

On the other hand, the formula in Theorem \ref{theorem of geometric formula for sm} holds for $L(Y_{1,-7}Y_{2,-4}Y_{3,-7})$. The module $K(m)$ has dimension 10 and is displayed in Figure \ref{example4}. In Figure \ref{example4}, almost all vertices carry a vector space of dimension 1, except the vertex $(2,-5)$ which carries a vector space of dimension 2. The maps associated with the arrows incident to $(2,-5)$ have the following matrices:
\[
\alpha=\begin{pmatrix} 1 & 1 \end{pmatrix}, \ \beta=\begin{pmatrix} 1 &  0 \end{pmatrix}, \ \gamma=\begin{pmatrix} 0 &  1 \end{pmatrix}, \ \delta=\begin{pmatrix} 1 \\ 0 \end{pmatrix},  \  \eta=\begin{pmatrix} 0 \\ 1 \end{pmatrix}. 
\]
All other arrows carry linear maps with $(\pm 1)$, whose sign is easily deduced from the defining relations of the Jacobian algebra $A$. There are 70 submodules in $K(m)$.

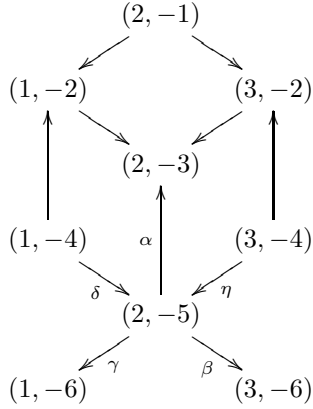
\begin{figure}
\resizebox{1.0\width}{1.0\height}{
\begin{minipage}[t]{.5\linewidth}
\centerline{
\begin{xy}
(10,100) *+{(1,-2)} ="0",
(25,110) *+{(2,-1)} ="1",
(25,90) *+{(2,-3)} ="2",
(40,100) *+{(3,-2)} ="3",
(10,60) *+{(1,-6)} ="4",
(25,70) *+{(2,-5)} ="5",
(40,80) *+{(3,-4)} ="6",
(40,60) *+{(3,-6)} ="7",
(10,80) *+{(1,-4)} ="9",
"1", {\ar"0"},
"1", {\ar"3"},
"0", {\ar"2"},
"3", {\ar"2"},
"6", {\ar^\eta "5"},
"5", {\ar^\gamma "4"},
"9", {\ar_\delta "5"},
"5", {\ar_\beta "7"},
"9", {\ar"0"},
"6", {\ar"3"},
"5", {\ar^\alpha "2"},
\end{xy}}
\end{minipage}}
\caption{The $A$-module $K(m)$ for $m=Y_{1,-7}Y_{2,-4}Y_{3,-7}$ in type $\mathbb{A}_3$.}\label{example4}
\end{figure}  

Then
\begin{align*}
\chi^-_q(L(m)) & = m((1+v_{1,-6}+v_{3,-6}+v_{1,-6}v_{3,-6})(1+v_{2,-3}+v_{2,-3}v_{1,-2}+v_{2,-3}v_{3,-2} \\
& + v_{1,-2}v_{2,-3}v_{3,-2}+v_{2,-3}v_{1,-2}v_{3,-2}v_{2,-1}) \\
& + (v_{1,-6}v_{2,-5}+v_{3,-6}v_{2,-5}) (v_{2,-3}+v_{2,-3}v_{1,-2}+v_{2,-3}v_{3,-2} + v_{1,-2}v_{2,-3}v_{3,-2} \\
& + v_{2,-3}v_{1,-2}v_{3,-2}v_{2,-1})  \\
& + v_{1,-6}v_{2,-3}v_{2,-5}v_{3,-2}v_{3,-4} (1+ v_{1,-2} + v_{1,-2}v_{2,-1}) \\
& + v_{3,-6}v_{2,-3}v_{2,-5}v_{1,-2}v_{1,-4} (1+ v_{3,-2} + v_{3,-2}v_{2,-1}) \\ 
& + v_{1,-6}v_{2,-3}v_{2,-5}v_{3,-2}v_{3,-4}v_{3,-6}(1+v_{1,-2}+v_{1,-2}v_{2,-1})  \\
& + v_{3,-6}v_{2,-3}v_{2,-5}v_{1,-2}v_{1,-4}v_{1,-6}(1+v_{3,-2}+v_{3,-2} v_{2,-1}) \\
& +v_{1,-6}v_{2,-5}v_{3,-6} +2v_{1,-6}v_{2,-3}v_{2,-5}v_{3,-6} + 2v_{1,-6}v_{2,-3}v_{2,-5}v_{3,-6}v_{1,-2} \\
& +2v_{1,-6}v_{2,-3}v_{2,-5}v_{3,-6}v_{3,-2}+2v_{1,-6}v_{2,-3}v_{2,-5}v_{3,-6}v_{1,-2}v_{3,-2} \\
& + 2v_{1,-6}v_{2,-3}v_{2,-5}v_{3,-6}v_{1,-2}v_{3,-2}v_{2,-1} \\  
& + v_{1,-6}v_{2,-3}v_{2,-5}v_{3,-6}v_{2,-5}(1+v_{1,-2}+v_{3,-2}+v_{1,-2}v_{3,-2}+v_{1,-2}v_{3,-2}v_{2,-1}) \\
& +v_{1,-6}v_{2,-3}v_{2,-5}v_{3,-6}v_{2,-5}v_{1,-2}v_{1,-4}(1+v_{3,-2}+v_{3,-2}v_{2,-1}) \\
& +v_{1,-6}v_{2,-3}v_{2,-5}v_{3,-6}v_{2,-5}v_{3,-2}v_{3,-4}(1+v_{1,-2}+v_{1,-2}v_{2,-1}) \\
& +v_{1,-6}v_{2,-3}v_{2,-5}v_{3,-6}v_{2,-5}v_{1,-2}v_{1,-4}v_{3,-2}v_{3,-4}  \\
& +v_{2,-3}v_{3,-2}v_{1,-2}v_{3,-6}v_{2,-5}v_{1,-4}v_{1,-6}v_{2,-5}v_{3,-4}v_{2,-1}),
\end{align*}
in agreement with Hernandez and Leclerc's conjectural geometric formula and its denominator is not square free. Indeed, there exists a submodule of $K(m)$ whose support at vertices $(2,-3)$, $(1,-6)$, $(3,-6)$ such that
\[
m v_{1,-6}v_{2,-3}v_{3,-6}=mA^{-1}_{1,-6}A^{-1}_{2,-3}A^{-1}_{3,-6} = \frac{z^2_{3,-3}z^2_{2,-6}z^2_{1,-3}z_{2,0}}{z_{1,-1}z_{1,-5}z_{2,-2}z^2_{2,-4}z_{3,-1}z_{3,-5}}. 
\] 

Moreover, the simple module $L(Y_{1,-7}Y_{2,-4}Y_{3,-7})$ is not special, because there are two dominant monomials $m=Y_{1,-7}Y_{2,-4}Y_{3,-7}$ and $mv_{1,-6}v_{2,-5}v_{3,-6}=Y_{2,-6}$ in $\chi_q(L(Y_{1,-7}Y_{2,-4}Y_{3,-7}))$. The simple module $L(Y_{1,-7}Y_{2,-4}Y_{3,-7})$ is not thin, because some terms in the $q$-character have coefficient $2>1$.

Using the method introduced in Section \ref{generic kernel}, we obtain the dimension vector $(d_{j,s}(K(m))_{(j,s)\in \mathbb{N}^{\Gamma^-_0}}$ of $K(m)$ as follows.
\begin{align*}
d_{j,s}(K(m))=\begin{cases}
1 & (j,s)=(2,-1), (1,-2), (3,-2), (2,-3), (1,-4), (3,-4),(1,-6), (3,-6), \\
2 & (j,s)=(2,-5), \\
0 & \text{ otherwise.}
\end{cases}
\end{align*}
\end{example}

Comparing Example \ref{example4.8} with Example \ref{example compare HL with MY}, we reformulate the following classification question.

In the cluster algebra $\mathscr{A}$, are cluster variables with square free denominators in bijection with prime snake modules and some other prime real modules whose truncated $q$-characters are not equal to their $q$-characters?  

\section{Factorial cluster algebras} \label{sect factorial}

In this section, we apply the results of \cite{ELS18} to show that the $\mathscr{C}_1$ cluster algebra is factorial for Dynkin types $\mathbb{A,D,E}$. 

Following Section 4.2 of \cite{HL10}, let $Q$ be a quiver with vertex set $\{1,\ldots,n,1',\ldots,n'\}$ subject to the following two conditions:
\begin{itemize}
\item[(1)] The full subquiver on $\{1,\ldots,n\}$ is an orientation of the associated Dynkin diagram $\Delta$ of type $\mathbb{A,D}$ or $\mathbb{E}$, oriented in such a way that every vertex in $I_0$ is a source and every vertex of $I_1$ is a sink;
\item[(2)] For every $i\in I$, one adds a frozen vertex $i'$ and an arrow $i'\to i$ if $i\in I_0$ and an arrow $i\to i'$ if $i\in I_1$.
\end{itemize}
Obviously, the defining quiver $Q$ is an acyclic quiver. Let $\mathscr{A}(Q)$ be the cluster algebra  defined by the initial seed $(\{x_1,\ldots,x_n,y_1,\ldots,y_n\},Q)$. Then $\mathscr{A}(Q)$ is the  $\mathscr{C}_1$ cluster algebra of type $\Delta$ in~\cite{HL10}. Let $x'_1,\ldots,x_n'$ be the $n$ cluster variables obtained from the initial seed by one single mutation. Then, for each $i$, we have $x_ix_i'=f_i$, where $f_i$ is a binomial in the initial seed. 
Recall from  \cite{ELS18} that  two vertices $i,j\in\{1,2,\ldots,n\}$   are called \emph{partners} if $f_i$ and $f_j$ have a non-trivial common factor. Partnership is an equivalence relation and the equivalence classes are called \emph{partner sets}. 
  
\begin{theorem}\label{factorial of C1}
The $\mathscr{C}_1$ cluster algebra is factorial for Dynkin types $\mathbb{A,D,E}$.
\end{theorem}
\begin{proof}
By Corollary 5.2 of \cite{ELS18}, we only need to show that every partner set in $Q$ is a singleton. This holds because for every $i\in I$ the exchange polynomial $f_i$ is a polynomial in the variable $y_i$.
\end{proof}

We give an example to explain Theorem \ref{factorial of C1}.
  
\begin{example} 
Let $\mathfrak{g}$ be of type $\mathbb{A}_3$. We choose $I_0=\{1,3\}$ and $I_1=\{2\}$. The quiver $Q$ is as follows.
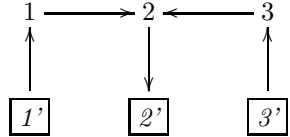
\begin{align*}
\xymatrix{
1 \ar[r] & 2 \ar[d] & 3 \ar[l] \\
\fbox{1'} \ar[u] & \fbox{2'} & \fbox{3'} \ar[u] }
\end{align*}
Here the vertices with boxes are frozen vertices and its associated exchange matrix is
 
\begin{align*}
\begin{bmatrix}
0 & 1 & 0 \\
-1 & 0 & -1 \\
0 & 1 & 0 \\
1 & 0 & 0 \\
0 & -1 & 0 \\
0 & 0 & 1
\end{bmatrix}.
\end{align*}
The exchange polynomials are 
\begin{align*}
& f_1 = x_2 + y_1, \\
& f_2 = x_1x_3 + y_2, \\
& f_3 = x_2 + y_3.
\end{align*}
The polynomials $f_1,f_2,f_3$ are pairwise coprime and hence every partner set in $Q$ is a singleton.
\end{example}

\begin{remark}
In Section 7.1 of \cite{GLS13}, Geiss, Leclerc, and Schr\"oer proved that the cluster algebra $\mathscr{A}$ associated to Dynkin type $\mathbb{A}_1$ is a factorial cluster algebra. They also showed that the cluster variables in a factorial cluster algebra are prime elements.  
In  \cite[Theorem 3.10]{ELS18}, it was shown that if $\mathscr{A}$ is a factorial cluster algebra and $x$ is a non-initial cluster variable, then the associated $F$-polynomial $F_x$ is prime. 

It is natural to ask whether the $\mathscr{C}_\ell$  cluster algebras, with $\ell>1$, are factorial. The argument in the proof of Theorem \ref{factorial of C1} does not work in this case, because we don't know whether these cluster algebras are  of acyclic type.
\end{remark}

\section*{Acknowledgements}
We would like to thank A. Garcia Elsener for explaining the results of \cite{ELS18} to us.

\end{document}